\documentclass[leqno,10pt]{elsarticle} 

\usepackage{amsmath,amsfonts,amsthm,mathrsfs,amssymb}
\usepackage{mathtools}
\usepackage[all]{xy}
\usepackage{enumerate}
\usepackage{graphicx}
\usepackage{bm}

\newtheorem{thm}{Theorem}
\newtheorem{cor}[thm]{Corollary}
\newtheorem{lem}[thm]{Lemma}
\newtheorem{prop}[thm]{Proposition}
\newdefinition{defn}[thm]{Definition}
\newdefinition{exmp}[thm]{Example}
\newdefinition{rem}[thm]{Remark}

\newcommand\dto{\dashrightarrow}
\newcommand\hto{\hookrightarrow}
\newcommand\lto{\longrightarrow}
\newcommand\nto{\stackrel}
\newcommand\lot{\longleftarrow}

\def\NN{{\bm{N}}}
\def\ZZ{{\bm{Z}}}
\def\kk{{\bm{k}}}
\def\PP{{\bm{P}}}
\def\RR{{\bm{R}}}
\def\AA{{\bm{A}}}

\newcommand\KKK{\mathscr K}
\newcommand\Sc{\mathscr{S}}
\newcommand\Hc{\mathscr{H}}
\newcommand\Zc{\mathcal{Z}}

\newcommand\Cc{\mathcal{C}}  
\newcommand\Tc{\mathscr{T}}

\DeclareMathOperator\Res{Res}
\DeclareMathOperator\conv{conv}
\DeclareMathOperator\Syz{Syz}

\def\endd{\mathrm{end}}
\DeclareMathOperator\coker{coker}
\DeclareMathOperator\Proj{Proj}

\DeclareMathOperator\Sym{Sym}

\DeclareMathOperator\ann{ann}

\DeclareMathOperator\Hom{Hom}

\DeclareMathOperator\im{im}
\DeclareMathOperator\length{length}
\DeclareMathOperator\res{Res}

\DeclareMathOperator\codim{codim}
\DeclareMathOperator\Spec{Spec}
\DeclareMathOperator\Biproj{Biproj}

\DeclareMathOperator\ext{Ext}

\DeclareMathOperator\Symm{Sym}
\newcommand\SIA{\Symm _A (I)}

\def\div{\textrm{div}}
\DeclareMathOperator\mat{Mat}
\DeclareMathOperator\relint{relint}

\newcommand\X{\textbf{X}}
\newcommand\T{\textbf{T}}

\def\s{\textbf{s}}

\def\h{\textbf{h}}
\def\x{\textbf{x}}

\newcommand\Nc{\mathcal{N}}

\newcommand\gen[1]{\left\langle#1\right\rangle}

\newcommand\paren[1]{\left(#1\right)}

\def\l.{\mathcal{L}_{\bullet}}
\def\ff.{\mathcal{F}_{\bullet}}
\def\a.{\mathcal{A}_\bullet}
\def\b.{\mathcal{B}_\bullet}

\def\k.{\mathcal{K}_{\bullet}}
\def\M.{\mathscr{M}_\bullet}
\def\Z.{\mathscr{Z}_\bullet}

\def\B.{\mathscr{B}_\bullet}

\def\LL{\mathcal L}

\def\UU{\mathcal U}

\def\LL{\mathcal L}
\def\BB{\mathcal B}

\def\KK{\mathcal K}

\def\EEE{\mathscr E}

\def\pp{\mathfrak{p}}
\def\mm{\mathfrak{m}}

\def\kkk{\kappa}

\def\pp{\mathfrak{p}}
\def\qq{\mathfrak{q}}
\def\mm{\mathfrak{m}}

\def\k.{\mathcal{K}_{\bullet}}
\def\BB{\mathcal B}
\def\KK{\mathcal K}
\def\P1{\PP^1}

\def\fxi{f_iY_i-g_iX_i}

\begin{document}
\begin{frontmatter}
\title{Compactifications of rational maps, and the implicit equations of their images.}
\author{Nicol\'{a}s Botbol\fnref{thanks}}
\address{Departamento de Matem\'atica\\
FCEN, Universidad de Buenos Aires, Argentina \\
\& Institut de Math\'ematiques de Jussieu \\
Universit\'e de P. et M. Curie, Paris VI, France \\
E-mail address: nbotbol@dm.uba.ar
}
\fntext[thanks]{This work was partially supported by the project ECOS-Sud A06E04, UBACYT X064, CONICET PIP 5617 and ANPCyT PICT
17-20569, Argentina.}
%
%
%
\begin{abstract}
In this paper we give different compactifications for the domain and the codomain of an affine rational map $f$ which parametrizes a hypersurface. We show that the closure of the image of this map (with possibly some other extra hypersurfaces) can be represented by a matrix of linear syzygies. 
We compactify $\AA^{n-1}$ into an $(n-1)$-dimensional projective arithmetically Cohen-Macaulay subscheme of some $\PP^N$.  One particular interesting compactification of $\AA^{n-1}$ is the toric variety associated to the Newton polytope of the polynomials defining $f$. We consider two different compactifications for the codomain of $f$: $\PP^n$ and $(\PP^1)^n$. In both cases we give sufficient conditions, in terms of the nature of the base locus of the map, for getting a matrix representation of its closed image, without involving extra hypersurfaces.
This constitutes a direct generalization of the corresponding results established in \cite{BuJo03}, \cite{BCJ06}, \cite{BD07}, \cite{BDD08} and~\cite{Bot08}.
\end{abstract}

\end{frontmatter}

\section{Introduction}\label{sec:intro}
The interest in computing explicit formulas for resultants and discriminants goes back to B\'ezout, Cayley, Sylvester and many others in the nineteenth century. It has been emphasized in the latest years due to the increase of computing power. Resultants have been related to the implicitization of a given rational maps; in turn, both resultants and discriminants can be seen as the implicit equation of a suitable map (cf.\ \cite{DFS07}). Lately, rational maps appeared in computer-engineering contexts, mostly applied to shape modeling using computer-aided design methods for curves and surfaces.
 
Rational algebraic curves and surfaces can be described in several different ways, the most common being parametric and implicit representations. Parametric representations describe the geometric object as the image of a rational map, whereas implicit representations describe it as the set of points verifying a certain algebraic condition, e.g.\ as the zeros of a polynomial equation. Both representations have a wide range of applications in Computer Aided Geometric Design (CAGD), and depending on the problem one needs to solve, one or the other might be better suited. It is thus interesting to be able to pass from parametric representations to implicit equations. This is a classical problem and there are numerous approaches to its solution, see \cite{SC95} and \cite{Co01} for a good historical overview. However, it turns out that the implicitization problem is computationally difficult. 

A promising alternative suggested in \cite{BD07} is to compute a so-called \textit{matrix representation} instead, which is easier to compute but still shares some of the advantages of the implicit equation. For a given hypersurface $\Hc \subset \PP^n$, a matrix $M$ with entries in the polynomial ring $\kk[X_0,\ldots,X_n]$ is called a {\slshape representation matrix} of $\Hc$ if it is generically of full rank and if the rank of $M$ evaluated in a point of $\PP^n$ drops if and only if the point lies on $\Hc$ (cf.\ \cite{BDD08}). Equivalently, a matrix $M$ represents $\Hc$ if and only if the greatest common divisor of all its minors of maximal size is a power of the homogeneous implicit equation $F \in \kk[X_0,\ldots,X_n]$ of $\Hc$. 

In the case of a planar rational curve $\Cc$ given by a parametrization of the form
$\AA^1 \stackrel{f}{\dashrightarrow} \AA^2$, $s \mapsto \paren{\frac{f_1(s)}{f_3(s)},\frac{f_2(s)}{f_3(s)}}$,
where $f_i \in \kk[s]$ are coprime polynomials of degree $d$ and $\kk$ is a field,
a linear syzygy (or moving line) is a linear relation on the polynomials $f_1,f_2,f_3$, i.e.\ a linear form $L = h_1X_1+h_2X_2+h_3X_3$ in the variables $X_1,X_2,X_3$ and with polynomial coefficients $h_i \in \kk[s]$ such that $\sum_{i=1,2,3} h_i f_i =0$. We denote by $\Syz (f)$ the set of all those linear syzygy forms and for any integer $\nu$ the graded part $\Syz(f)_\nu$ of syzygies of degree at most $\nu$. To be precise, one should homogenize the $f_i$ with respect to a new variable and consider $\Syz (f)$ as a graded module here. It is obvious that $\Syz(f)_\nu$ is a finite-dimensional $\kk$-vector space with basis $(L_1,\ldots,L_k)$ ($k=k(\nu)$) obtained by solving a linear system. If $L_i=\sum_{|\alpha|=\nu}s^\alpha L_{i,\alpha}(X_1,X_2,X_3)$, we define the matrix $M_\nu=(L_{i,\alpha})_{1\leq i\leq k, |\alpha|=\nu}$, that is, the coefficients of the $L_i$ with respect to a $\kk$-basis of $\kk[s]_\nu$ form the  columns of the matrix. Note that the entries of this matrix are linear forms in the variables $X_1,X_2,X_3$ with coefficients in the field $\kk$. Let $F$ denote the homogeneous implicit equation of the curve and $\deg(f)$ the degree of the parametrization as a rational map. Intuitively, $\deg(f)$ measures how many times the curve is traced. It is known that for $\nu \geq d-1$, the matrix $M_\nu$ is a representation matrix; more precisely: if $\nu=d-1$, then $M_\nu$ is a square matrix, such that $\det(M_\nu)=F^{\deg(f)}$. Also, if $\nu \geq d$, then $M_\nu$ is a non-square matrix with more columns than rows, such that the greatest common divisor of its minors of maximal size equals $F^{\deg(f)}$. In other words, one can always represent the curve as a square matrix of linear syzygies. One could now actually calculate the implicit equation.

For surfaces, matrix representations have been studied in the recent article \cite{BDD08} for the case of $2$-dimensional projective toric varieties. Previous works had been done in this direction, with two main approaches: One allows the use of quadratic syzygies (or higher-order syzygies) in addition to the linear syzygies, in order to be able to construct square matrices, the other one only uses linear syzygies as in the curve case and obtains non-square representation matrices.

The first approach using linear and quadratic syzygies (or moving planes and quadrics) has been treated in \cite{Co03} for
base-point-free homogeneous parametrizations and some genericity assumptions, when $\Tc=\PP^2$. The authors of \cite{BCD03} also treat the case of toric surfaces in the presence of base points. In \cite{AHW05}, square matrix representations of bihomogeneous parametrizations, i.e.\ $\Tc=\PP^1 \times \PP^1$, are constructed with linear and quadratic syzygies, whereas \cite{KD06} gives such a construction for parametrizations over toric varieties of dimension 2. The methods using quadratic syzygies usually require additional conditions on the parametrization and the choice of the quadratic syzygies is often not canonical. 

The second approach, even though it does not produce square matrices, has certain advantages, in particular in the sparse setting that we present. In previous publications, this approach with linear syzygies, which relies on the use of the so-called approximation complexes has been developed in the case $\Tc=\PP^2$, see for example \cite{BuJo03}, \cite{BC05}, and \cite{Ch06}, and in \cite{BD07} for  bihomogeneous parametrizations of degree $(d,d)$. However, for a given affine parametrization $f$, these two varieties are not necessarily the best choice of a compactification, since they do not always reflect well the combinatorial structure of the polynomials $f_1,\ldots,f_4$. The method was extended to a much larger class of varieties, namely toric varieties of dimension $2$ (cf.\ \cite{BDD08}), where it was shown that it is possible to choose a ``good'' toric compactification of $(\AA^*)^2$ depending on the polynomials $f_1,\ldots,f_4$, which makes the method applicable in cases where it failed over $\PP^2$ or $\PP^1 \times \PP^1$ and also, that it is significantly more efficient leading to smaller representation matrices. 

\medskip

In this article we extend the method of computing an implicit equation of a parametrized hypersurface focusing on different compactifications of the domain $\Tc$ and of the codomain $\PP^n$ and $(\P1)^n$. Hereafter we will always assume that $\Tc$ is embedded in $\PP^N$, and its coordinate ring $A$ is $n$-dimensional, graded and Cohen-Macaulay.

\medskip

In Section \ref{sec3Pn} we focus on the implicitization problem for a rational map $\varphi:\Tc \dto \PP^n$ defined by $n+1$ polynomials of degree $d$. We extend the method for projective $2$-dimensional toric varieties developed in \cite{BDD08} to a map defined over an $(n-1)$-dimensional arithmetically Cohen-Macaulay closed scheme $\Tc$ embedded in $\PP^N$. We show that we can relax the hypotheses on the base locus by allowing it to be a zero-dimensional almost locally complete intersection scheme.

\medskip

In order to consider more general parametrizations given by rational maps of the form $f=\paren{\frac{f_1}{g_1},\hdots,\frac{f_n}{g_n}}$ with different denominators $g_1,\hdots,g_n$, we develop in Section \ref{sec4multiproj} the study of the $(\PP^1)^n$ compactification of the codomain.  With this approach, we generalize, in the spirit of \cite{Bot08}, the method of implicitization of projective hypersurfaces embedded in $(\P1)^n$ to general hypersurfaces parametrized by any $(n-1)$-dimensional arithmetically Cohen-Macaulay closed subscheme of $\PP^N$. As in the previously mentioned works, we compute the implicit equation as the determinant of a complex which coincides with the gcd of the maximal minors of the last matrix of the complex, and we give a depth study of the geometry of the base locus.

\medskip

Section \ref{sec5algorithmic} is devoted to the algorithmic approach of  both cases studied in Section \ref{sec3Pn} and \ref{sec4multiproj}. We show how to compute the dimension of the representation matrices obtained in both cases by means of the Hilbert functions of the ring $A$ and its Koszul cycles. In the last part of this section, we show, for the case of toric  parametrizations given from a polytope $\Nc(f)$ (cf.\ \ref{defNf}), how the interplay between multiples of $\Nc(f)$ and degree of the maps may lead to have smaller matrices.

\medskip

We conclude by giving in Section \ref{sec6examples} several examples. First, we show in a very sparse setting the advantage of not considering the homogeneous compactification of the domain when denominators are very different. In the second example, we extend this idea to the case of a generic affine rational map in dimension $2$ with fixed Newton polytope.
In the last example we give, for a parametrized toric hypersurface of $(\P1)^n$, a detailed analysis of the relation between the nature of the base locus of a map and the extraneous factors appearing in the computed equation.


\section{General setting} \label{sec2setting}\label{ImageCodim1}

Throughout this section we will give a general setting for the implicitization problem of hypersurfaces. Our aim is to analyze how far these techniques from homological commutative algebra (syzygies and graded resolutions) can be applied.

\medskip

Write $\AA^k:= \Spec (\kk[T_1,\hdots,T_k])$ for the $k$-dimensional affine space over $\kk$. Assume we are given a rational map
\begin{equation}\label{initsetting}
f: \AA^{n-1} \dto \AA^n : \s:=(s_1,\hdots,s_{n-1})  \mapsto \paren{\frac{f_1}{g_1},\hdots,\frac{f_n}{g_n}}(\s)
\end{equation}
where $\deg(f_i)=d_i$ and $\deg(g_i)=e_i$ without common factors. Observe that this setting is general enough to include all classical implicitization problems. Typically all $g_i$ are assumed to be equal and a few conditions on the degrees are needed, depending on the context. 

We consider a rational map $\psi:\Tc \dto \KKK$, where $\Tc$ and $\KKK$ are suitable compactifications of a suitable dense open subset of $\AA^{n-1}$ and $\AA^n$ respectively, in such a way that the map $f$ extends through $\Tc$ and $\KKK$ via $\psi$ and that the closed image of $f$ can be recovered from the closed image of $\psi$. 

Assume $\Tc$ can be embedded in some $\PP^N$, and set $A$ for the homogeneous coordinate ring of $\Tc$. Since $\AA^{n-1}$ is irreducible, so is $\Tc$, hence $A$ is a domain. Assume also that $\Tc$ defines via $\psi$ a hypersurface in $\KKK$, hence, $\ker(\psi^*)$ is a principal ideal, generated by the implicit equation. 

Most of our results are stated for a general arithmetically Cohen-Macaulay scheme as domain. Nevertheless, the map \eqref{initsetting} gives rise, naturally, to a toric variety $\Tc$ on the domain (cf.\ \cite[Sect.\ 2]{KD06}, \cite{Co03b}, and \cite[Ch.\ 5 \& 6]{GKZ94}) associated to the following polytope $\Nc(f)$.

\begin{defn}\label{defNf}
Let $f$ denote a map as in Equation \eqref{initsetting}. We will write 
\[
 \Nc(f):=\conv\paren{\bigcup_{i=1}^n \paren{\Nc(f_i)\cup \Nc(g_i)}}
\]
the convex hull of the union of the Newton polytopes of all the polynomials defining the map $f$.
\end{defn}

There is a standard way of associating a semigroup $S_\Nc$ to a polytope $\Nc\subset \RR^{n-1}$: take $\iota: \RR^{n-1}\hto \RR^n: x\mapsto (x,1)$, and define $S_\Nc$ as the semigroup generated by the lattice points in $\iota(\Nc)$. Due to a theorem of Hochster, if $S_\Nc$ is normal then the semigroup algebra $\kk[S_\Nc]$ is Cohen-Macaulay. Unluckily, it turns out that $S_\Nc$ is in general not always normal. A geometric or combinatorial characterization of the normality of $\kk[S_\Nc]$ is one of the most important open problem in combinatorial algebra (cf.\ \cite{BGN97}). 

Note that $m\Nc\times \{m\}=\{(p_1+\cdots+p_m,m)\ :\ p_i\in\Nc\}\subset S_\Nc\cap (\ZZ^{n-1}\times \{m\})$ for any $m\in \NN$, but in general these two sets are not equal.
When this happens for all $m\in \NN$, we say that the polytope $\Nc$ is normal, equivalently $(m\cdot \Nc) \cap \ZZ^{n-1} = m \cdot (\Nc \cap  \ZZ^{n-1})$ for all $m \in \NN$, and in this case it follows that $\kk[S_\Nc]$ is Cohen-Macaulay. 

In this article we focus on the study of toric varieties by fixing an embedding. Changing $\Nc$ by a multiple $l\cdot \Nc$ changes the embedding, hence, we will fix the polytope. Since we also need Cohen-Macaulayness of the quotient ring by the corresponding toric ideal in several results, we will assume throughout that $\Nc$ is normal.

\begin{rem}\label{NfAlwaysNormal}
Given a map $f$ as in Equation \eqref{initsetting}, we will always assume that $\Nc:=\Nc(f)$ is normal. Therefore, the coordinate ring $A$ of $\Tc$ will be always Cohen-Macaulay, hence $\Tc\subset \PP^N$ will be \emph{arithmetically Cohen-Macaulay (aCM)}. This is automatic when $n=2$.
\end{rem}

The polytope $\Nc(f)$ defines a $(n-1)$-dimensional projective toric variety $\Tc$ provided with an ample line bundle which defines an embedding: if $N=\#(\Nc(f) \cap \ZZ^{n-1})-1$ we have $\Tc \subseteq \PP^N$  (cf.\ \cite{Co03b}). Write $\Nc'(f)$ for the smallest lattice contraction of $\Nc(f)$ (that is $\Nc(f)=d\Nc'(f)$ where $\Nc'(f)$ is a lattice polytope and $d$ is maximal with this property), and $\rho$ for the embedding determined by this ample sheaf. We get that the map
\begin{equation}
 (\AA^*)^{n-1}  \stackrel{\rho}{\hookrightarrow} \PP^{N'} : (\s) \mapsto (\ldots : \s^\alpha  : \ldots),
\end{equation}
where $\alpha \in \Nc'(f) \cap \ZZ^{n-1}$ and $N'=\#(\Nc'(f) \cap \ZZ^{n-1})-1$, which factorizes $f$ through a rational map $\psi: \Tc\dto \KKK$. 

The main reason for considering projective toric varieties associated to the Newton polytope $\Nc(f)$ of $f$, is based on the following fact.

\begin{rem}
 Assume $f$ is as in Equation \eqref{initsetting}, with $g_1=\cdots =g_n$. Write $f_0:=g_i$ for all $i$. Assume also that all $f_i$ are generic with Newton polytope $\Nc$, and hence write $\Nc:=\Nc(f_i)$ for all $i$. Set $N:=\#(\Nc \cap \ZZ^{n-1})-1$ and let $\Tc \subset \PP^{N}$ be the toric variety associated to $\Nc$. Write $\phi: \Tc \dto \PP^n: \T \mapsto (h_0:\cdots:h_n)$ the map induced by $f$. Since the coefficients are generic, the point associated to these coefficients is not in $V(\Res_\Nc(h_0,\hdots,h_n))$; where $V(\Res_\Nc(h_0,\hdots,h_n))$ stands for the zeroes locus of the sparse resultant $\Res_\Nc(h_0,\hdots,h_n)$ associated to $h_0,\hdots,h_n$. Hence, they have no common root in $\Tc$. 
Thus, $\phi$ has empty base locus in $\Tc$.

If we take instead another lattice polytope $\tilde{\Nc}$ strictly containing $\Nc$, the $f_i$ will not be generic relative to $\tilde{\Nc}$, and typically the associated map $\tilde{\phi}$ will have a non-empty base locus in the toric variety $\tilde{\Tc}$ associated to $\tilde{\Nc}$.
\end{rem}

\medskip

Henceforward, let $\Tc$ be any $(n-1)$-dimensional projective arithmetically Cohen-Macaulay scheme provided with an embedding $\rho$ into some $\PP^N$. Assume $\KKK:=\PP^n$ or $\KKK:=(\P1)^n$. In both cases we consider the diagram 
\[
 \Tc \nto{\pi_1}{\lot} \Tc \times \KKK \nto{\pi_2}{\lto} \KKK,
\]
where $\pi_1$ and $\pi_2$ are the natural projections. Since $\psi:\Tc \dto \KKK$ is not, in principle, defined everywhere in $\Tc$, we set $\Omega$ for the open set of definition of $\psi$. Precisely, we define

\begin{defn}\label{defXyOmega}
 Assume $\KKK:=\PP^n$, in this setting, $\psi$ will be denoted by $\varphi$. Hence, write $\varphi:\Tc \dto \PP^n$ given by $\s\mapsto (h_0:\cdots:h_{n})(\s)$, then define the base locus of $\varphi$ as the closed subscheme of $\Tc$
\[
 X_{\PP^n}:=\Proj\paren{A/(h_0,\hdots,h_{n})}.
\]
Similarly, for $\KKK:=(\P1)^n$, $\psi$ will be denoted by $\phi$. Hence, write $\phi:\Tc \dto (\P1)^n$ given by $\s\mapsto (f_1:g_1)(\s)\times\cdots\times(f_n:g_n)(\s)$, then define the base locus of $\phi$ as the closed subscheme of $\Tc$
\[
 X_{(\P1)^n}:=\Proj\paren{A/\prod_i(f_i,g_i)}.
\]
In the sequel $X$ will stand for $X_{\PP^n}$ and $X_{(\P1)^n}$, and will be understood, depending on the context. In any case, we call $\Omega$ the complement of the base locus, namely $\Omega:=\Tc \setminus X$. Let $\Gamma_\Omega$ be the graph of $\phi$ or $\varphi$ into $\subset \Omega \times \KKK$. 
\end{defn}

 Clearly $\Gamma_\Omega \nto{\pi_1}{\lto} \Omega$ is birational, which is in general not the case over $X$. As was shown in \cite{Bot08}, the scheme structure of the base locus when $X=X_{(\P1)^n}$ can be fairly complicated and extraneous factors may occur when projecting on $(\P1)^n$ via $\pi_2$ (cf.\ Section $4.3$). This motivates the need for a splitting of the base locus, giving rise to families of multiprojective bundles over $\Tc$.

Due to this important difference between the projective and multiprojective case, we need to separate the study of the two settings. In the next section, we treat the case $\KKK:=\PP^n$, and in Section $4$ the case $\KKK:=(\P1)^n$. In both situations, we find a matrix representation of the closed image of $\varphi$ and $\phi$, and we compute the implicit equation and extraneous factors that occur.

\medskip

\section{The case $\KKK=\PP^n$}\label{sec3Pn}

In this section we focus on the computation of the implicit equation of a hypersurface in $\PP^n$, parametrized by an $(n-1)$-dimensional arithmetically Cohen Macaulay (aCM) subscheme of some projective space $\PP^N$. We generalize \cite{BuJo03}, \cite{BCJ06}, \cite{BDD08}, et.\ al., and we give a more general result on the acyclicity of the approximation complex of cycles, by relaxing conditions on the base ring and on the base locus. 

\medskip
Henceforward in this section, let $\Tc$ be a $(n-1)$-dimensional projective aCM closed scheme over a field $\kk$, embedded in $\PP^N_{\kk}$, for some $N\in \NN$. Write $A=\kk[T_0,\hdots,T_N]/J$ for its CM coordinate ring, and $J$ the homogeneous defining ideal of $\Tc$. Set $\T:= T_0,\hdots,T_N$ the variables in $\PP^N$, and $\X$ the sequence $X_0,\hdots,X_n$ of variables in $\PP^n$.

We denote $\mm:=A_+=(\T)\subset A$, the maximal homogeneous irrelevant ideal of $A$.

Let $\varphi$ be a finite map defined over a relative open set $U$ in $\Tc$ defining a hypersurface in $\PP^n$, e.g.\ $U=\Omega$:
\begin{equation}\label{eqSettingPn}
 \PP^N \supset \Tc \nto{\varphi}{\dto} \PP^n : \T \mapsto (h_0:\cdots:h_n)(\T),
\end{equation}
where $h_0,\hdots,h_n$ are homogeneous elements of $A$ of degree $d$. Set $\h:=h_0,\hdots,h_n$. The map $\varphi$ gives rise to a morphism of graded $\kk$-algebras in the opposite sense
\begin{equation}
 \kk[X_0,\hdots,X_n] \nto{\varphi^\ast}{\lto} A: X_i \mapsto h_i(\T).
\end{equation}
Since $\ker(\varphi^\ast)$ is a principal ideal in $\kk[\X]$, write $H$ for a generator. We proceed as in \cite{BuJo03} or in \cite{BDD08} to get a matrix such that the gcd of its maximal minors gives $H^{\deg(\phi)}$, or possibly, a multiple of it.

\begin{rem}\label{RemToricCasePn}
Observe that if we start with an affine setting as in \eqref{initsetting}, $\Tc\subset \PP^N$ can be taken as the embedded toric variety associated to $\Nc'(f)$. In the classical implicitization problem it is common to suppose that $g_i=g_j$ for all $i$ and $j$, and $\deg(f_i)=\deg(g_i)=d$ for all $i$. Hence write $f_0$ for any of the $g_i$.  This setting gives naturally rise to a homogeneous compactification of the codomain, defined by the embedding
\begin{equation}
\AA^n \nto{j}{\hto} \PP^n: \x \mapsto (1:\x).
\end{equation}
It is clear that for $f_0,\hdots,f_n$ taken as above, the map $f: \AA^{n-1}\dto \AA^n$ of equation \eqref{initsetting} compactifies via $\rho$ and $j$ to $\varphi:\Tc \dto \PP^n$. It is important to note that $\overline{\im(f)}$ can be obtained from $\overline{\im(\varphi)}$ and vice-versa, via the classical (first variable) dehomogenization and homogenization respectively. Finally, we want to give a matrix representation for a toric hypersurface of $\PP^n$ given as the image of the toric rational map $\varphi: \Tc \dto \PP^n : \T \mapsto (h_0:\cdots:h_n)(\T)$.
\end{rem}

Next, we introduce the homological machinery needed to deal with the computations of the implicit equations and the representation matrix of the hypersurface.

\subsection{Homological algebra tools}

For simplicity, we denote by $T_i$ the classes of each variable in the quotient ring $A=\kk[\T]/J$. Recall that $A$ is canonically graded, each variable having weight $1$. Let $I=(h_0,\hdots,h_{n}) \subset A$ be the ideal generated by the $h_i$'s.

More precisely, we will see that the implicit equation of $\Sc$ can be recovered as the determinant of certain graded parts of the $\Zc$-complex we define below. We denote by ${\Zc}_{\bullet}$ the approximation complex of cycles associated to the sequence $h_0,\hdots,h_{n}$ over $A$ (cf.\ \cite{Va94}), as in the Definition \ref{defZcomplex}. 

\medskip
 Consider the Koszul complex $(K_\bullet(\h,A),\delta_\bullet)$ associated to $h_0,\ldots,h_{n}$ over $A$ and denote $Z_i=\ker(\delta_i)$, $B_i=\im(\delta_{i+1})$. It is of the form
\[
K_\bullet(\h,A): \quad  A[-(n+1)d] \nto{\delta_{n+1}}{\lto} A[-nd]^{n+1} \nto{\delta_n}{\lto} \cdots \nto{\delta_2}{\lto} A[-d]^{n+1} \nto{\delta_1}{\lto} A
\]
where the differentials are matrices such that every non-zero entry is $\pm h_i$. 

Write $K_i:= \bigwedge^iA^{n+1}[-i\cdot d]$. Since $Z_i\subset K_i$, it keeps the shift in the degree. Note that with this notation the sequence 
\begin{equation}\label{sesZKB}
 0\to Z_i\to K_i\to B_{i-1}\to 0
\end{equation}
is exact graded, and no degree shift is needed.

We introduce new variables $X_0,\ldots,X_{n}$ with $\deg(X_i)=1$. Since $A$ is $\NN$-graded, $A[\X]$ inherits a bigrading. 

\begin{defn}\label{defZcomplex}
Denote by $\Zc_i= Z_i[i \cdot d] \otimes_A A[\X]$ the ideal of cycles in $A[\X]$, and write $[-]$ for the degree shift in the variables $T_i$ and $(-)$ the one in the $X_i$. The approximation complex of cycles $(\Zc_\bullet(\h,A),\epsilon_\bullet)$, or simply $\Zc_\bullet$, is the complex 
\begin{equation}\label{CompAppZ}
\Zc_\bullet(\h,A): \quad 0 \lto \Zc_n(-n) \nto{\epsilon_n}{\lto} \Zc_{n-1}(-(n-1)) \nto{\epsilon_{n-1}}{\lto} \cdots \nto{\epsilon_2}{\lto} \Zc_1(-1) \nto{\epsilon_1}{\lto} \Zc_0
\end{equation}
where the differentials $\epsilon_\bullet$ are obtained by replacing $h_i$ by $X_i$ for all $i$ in the matrices of $\delta_\bullet$. 
\end{defn}

It is important to observe that $H_0(\Zc_\bullet) = A[\X]/\im(\epsilon_1) \cong \SIA$. Note that the degree shifts indicated in the complex above are with respect to the grading given by the $X_i$'s, while the degree shifts with respect to the grading of $A$ are already contained in our definition of the $\Zc_i$'s. From now on, when we take the degree $\nu$ part of the approximation complex, denoted $(\Zc_\bullet)_\nu$, it should always be understood to be taken with respect to the grading induced by $A$.

\medskip

Under certain conditions on the base locus of the map, this complex is acyclic and provides a free $\kk[\X]$-resolutions of $(\SIA)_\nu$ for all $\nu$. Hence, we focus on finding acyclicity conditions for the complex ${\Zc}_{\bullet}$. In this direction we have

\begin{lem}\label{ZacycALCI} 
Let $m\geq n$ be non-negative integers, $A$ an $m$-dimensional graded Cohen-Macaulay ring and $I=(h_0,\ldots,h_{n}) \subset A$ is of codimension (hence depth) at least $n-1$ with $\deg(h_i)=d$ for all $i$. Assume that $X:=\Proj (A/I) \subset \Sc$ is locally defined by $n$ equations (i.e.\ locally an almost complete intersection). Then ${\Zc}_{\bullet}$ is acyclic.
\end{lem}
\begin{proof}
The proof follows ideas of \cite[Lemma 2]{BC05} and \cite[Lemma 1]{BD07}. Observe that the lemma is unaffected by an extension of the base field, so one may assume that $\kk$ is infinite.

By \cite[{Theorem} 12.9]{HSV}, we know that $\Zc_\bullet$ is acyclic (resp.\ acyclic outside $V(\mm)$) if and only if $I$ is generated by a proper sequence (resp.\ $X$ is locally defined by a proper sequence). Recall that a sequence $a_1,\ldots,a_n$ of elements in a commutative ring $B$ is a \emph{proper sequence} if $a_{i+1}H_{j}(a_1,\ldots,a_i;B)=0$ for $i=0,\ldots,n-1$ and $j>0$, where the $H_j$'s denote the homology groups of the corresponding Koszul complex.

By following the same argument of \cite[{Lemma} 2]{BC05} and since $X$ is locally defined by $n$ equations, one can choose $\tilde h_0,\ldots,\tilde h_{n}$ to be sufficiently generic linear combinations of the $h_i$'s such that 
\begin{enumerate}
 \item $(\tilde h_0,\ldots,\tilde h_{n})=(h_0,\ldots,h_{n}) \subset A$,
 \item $\tilde h_0,\ldots,\tilde h_{n-2}$ is an $A$-regular sequence, hence $\tilde h_0\ldots,\tilde h_{n-1}$ is a proper sequence in $A$,
 \item $\tilde h_0,\ldots,\tilde h_{n-1}$ define $X$ in codimension $n-1$.
\end{enumerate}

Note that this last condition is slightly more general (and coincides when $m=n$) than the one in \cite[{Lemma} 2]{BC05}. Set $J:=(\tilde h_0,\dots,\tilde h_{n-1})$ and write $J^{um}$ for the unmixed part of $J$ of codimension $n-1$. Hence, observe that we obtain $\tilde h_{n} \in J^{um}$. 

\medskip

Since $\tilde h_{n}\in J^{um}$, we show that $\tilde h_{n}H_1(\tilde h_0,\dots,\tilde h_{n-1};A)=0$. Applying \cite[Thm.\ 1.6.16]{BH} to the sequence $\tilde h_0,\hdots, \tilde h_{n-1}$, we obtain that $H_1(\tilde h_0,\dots,\tilde h_{n-1};A)\cong \ext^{n-1}_A(A/J,A)$. Taking the long exact sequence of $\ext^\bullet_A(-,A)$ coming from the short exact sequence $0\to J^{um}/J\to A/J\to A/J^{um}\to 0$, we get that
\begin{equation*}
 \xymatrix@R-16pt{
  \cdots \ar[r]
   & \ext^{n-2}_A(J^{um}/J,A) \ar[r] 
   & \ext^{n-1}_A(A/J,A)\ar `[r] `[d] '[dll] *{} `[ddll] `[ddll] [ddl]
   & \\
 &&& \\
 {}
   & \ext^{n-1}_A(A/J^{um},A) \ar[r]
   & \ext^{n-1}_A(J^{um}/J,A) \ar[r]
   & \cdots
}
\end{equation*}
is exact. Since $A$ is a Cohen-Macaulay noetherian graded ring, and $J^{um}/J$ is a $m-(n-1)$-dimensional $A$-module, $\ext^{n-1}_A(J^{um}/J,A)$ and $\ext^{n-2}_A(J^{um}/J,A)$ vanish (cf.\ \cite[Thm.\ 17.1]{Mats}). Hence
\[
 \ext^{n-1}_A(A/J,A)\cong \ext^{n-1}_A(A/J^{um},A),
\]
thus, since $\tilde h_{n}\in J^{um}$, $\tilde h_{n}$ annihilates $\ext^{n-1}_A(A/J^{um},A)$, hence also $\tilde h_n$ annihilates $H_1(\tilde h_0,\hdots,\tilde h_{n-1};A)$ which finishes the proof. 
\end{proof}

We stress in the following remark one useful application of the previous Lemma \ref{ZacycALCI}.

\begin{rem}\label{remAcycZPmPn}
Let $m\geq n$ be non-negative integers. Set $\Tc$ a arithmetically Cohen-Macaulay scheme over $\kk$ embedded in some $\PP^N$ with coordinate ring $A$ of affine dimension $m$. Assume we are given a rational map $\varphi:\Tc \dto \PP^n$ given by $n+1$ homogeneous polynomials $h_0,\hdots,h_n\in A:=\kk[T_0,\hdots,T_N]/I(\Tc)$. Write $\Zc_\bullet$ for the approximation complex of cycles associated to the sequence $h_0,\hdots,h_n$. If the base locus of $\varphi$, $X\subset \Tc$, is locally defined by $n$ equations, then $\Zc_\bullet$ is acyclic, independent of $m$ and $N$.
\end{rem}

We translate Lemma \ref{ZacycALCI} geometrically.

\begin{cor}\label{corAcycZToricPn}
Assume $m = n$ be a non-negative integer. Let $\Tc$ be an $(n-1)$-dimensional CM closed subscheme of $\PP^N$ defined by a homogeneous ideal $J$, and coordinate ring $A=\kk[\T]/J$. Assume we are given a rational map $\varphi:\Tc \dto \PP^n$ given by $n+1$ homogeneous polynomials $h_0,\hdots,h_n\in A$ of degree $d$. Write $\Zc_\bullet$ for the approximation complex of cycles associated to the sequence $h_0,\hdots,h_n$. If the base locus of $\varphi$, $X\subset \Tc$, is finite, and locally an almost complete intersection (defined by $n$ equations), then $\Zc_\bullet$ is acyclic.
\end{cor}

The following result establishes a vanishing criterion on the graded strands of the local cohomology of $\Sym_A(I)$, which ensures that the implicit equation can be obtained as a generator of the annihilator of the symmetric algebra in that degree.

Since $A$ is a finitely generated graded Cohen Macaulay $A$-module of dimension $n$, $H^i_\mm(A)=0$ for all $i \neq n$ and $H^n_\mm(A) =\omega_A^\vee$, where $(-)^\vee:=\ ^*\hom_A(-,\kk)$ stands for the Matlis dualizing functor (cf.\ \cite{BH}). Write $a_i(M):=\inf \{\mu\ : (H^i_\mm(M))_{>\mu}=0\}$. Hence, we set $\gamma:= a_n(A)=\inf\{\mu\ :\ (\omega_A^\vee)_\mu=0\}$, and we conclude the following result.

\begin{thm}\label{annih} 
Let $A=\kk[\T]/J$ be a CM graded ring of dimension $n$. Let $I=(h_0,\hdots,h_n)$ be a homogeneous ideal of $A$, with $\deg(h_i)=d$ for all $i$. Let $X:=\Proj(A/I) \subset \Tc$ be finite and locally an almost complete intersection. Set $\nu_0 := \max\{(n-2)d, (n-1)d-\gamma\}$, then $H^0_\mathfrak{m} ( \Sym_A(I) )_\nu =0$ for all $\nu\geq \nu_0$.
\end{thm}
\begin{proof}
For the bound on $\nu$, consider the two spectral sequences associated to the double complex $C^\bullet_\mm(\Zc_\bullet)$, both converging to the hypercohomology of $\Zc_\bullet$. The first spectral sequence stabilizes at step two with
\[
_\infty'E^p_q =\ _2'E^p_q = H^p_\mm(H_q(\Zc_\bullet)) = \left\lbrace\begin{array}{ll}H^p_\mm(\SIA) & \mbox{for }q=0, \\
0 & \mbox{otherwise.} \end{array}\right.
\]
The second has first terms $_1''E^p_q =\ H^p_\mm(Z_q)[qd]\otimes_A A[\X](-q)$. The comparison of the two spectral sequences shows that $H^0_\mm(\Sym_A(I))_\nu$ vanishes as soon as $(_1{''}E^{p}_p)_\nu$ vanishes for all $p$, in fact we have that 
\[
 \endd(H^0_\mm(\Sym_A(I)))\leq \max_{p\geq 0}\{\endd(_1{''}E^{p}_p)\}=\max_{p\geq 0}\{\endd(H^p_\mm(Z_p))-pd\},
\]
where we denote, for an $A$-module $M$, $\endd(M)= \max \{ \nu \ | \ M_\nu \neq 0 \}$.
Since $Z_0\cong A$ we get $H^0_\mm(Z_0)=0$. The sequence $ 0 \to Z_{i+1} \to K_{i+1} \to B_i \to 0$ is graded exact (cf.\ Equation \eqref{sesZKB}), hence, from the long exact sequence of local cohomology for $i=0$ (writing $B_0=I$) we obtain
\[
  \cdots \to H^0_\mm(I) \to H^1_\mm(Z_{1}) \to H^1_\mm(K_{1}) \to \cdots .
\]
As $I$ is an ideal of an integral domain, $H^0_\mm(I)=0$, it follows from the local cohomology of $A$ that $H^1_\mm(K_{1})=0$, hence $H^1_\mm(Z_{1})$ vanishes.
By construction, $Z_{n+1}=0$ and that $B_{n}=\im(d_{n})\simeq A[-d]$. Using the fact that $H^{n}_\mathfrak{m}(A)_\nu =0$ for $\nu \geq -1$ (resp.\ $\nu \geq 0$), we can deduce that $H^{n}_\mathfrak{m}(Z_{n})_\nu=H^{n}_\mathfrak{m}(B_{n})_\nu = (\omega_A^\vee)[d]= 0$ if $\nu \geq d-\gamma$.
Write 
\[
 \epsilon_p:= \endd(_1{''}E^{p}_p)=\endd(H^p_\mm(Z_p))-pd 
\]
By \cite[Cor.\ 6.2.v]{Ch04} $\endd(H^p_\mm(Z_p))\leq  \max_{0\leq i\leq n-p} \{a_{p+i}(A) + (p+i+1)d\}=\max\{nd, (n+1)d-\gamma\}$, where $\gamma:=-a_{n}(A)$ as above. Hence, $ \epsilon_p:= \max\{(n-p)d, (n+1-p)d-\gamma\}$. As $\epsilon_p$ decreases when $p$ increases, $ \epsilon_p\leq \epsilon_2= \max\{(n-2)d, (n-1)d-\gamma\}$ which completes the proof.
\end{proof} 

This generalizes \cite{BuJo03}, \cite{BCJ06} and \cite{BDD08} to general $(n-1)$-dimensional arithmetically Cohen-Macaulay schemes. Next, we recall how the homological tools developed in this part are applied for computing the implicit equation of the closed image of a rational map.

\subsection{The representation matrix, the implicit equation, and the extraneous factor}
It is well known that the annihilator above can be computed as the determinant (or MacRae invariant) of the complex $(\Zc_\bullet)_{\nu_0}$ (cf.\ \cite{BuJo03}, \cite{BCJ06}, \cite{Bot08}, \cite{BDD08} et al.). Hence, the determinant of the complex $(\Zc_\bullet)_{\nu_0}$ is a multiple of a power of the implicit equation of $\Sc$. Indeed, we conclude the following result.

\begin{lem}\label{lemAnnih} 
Let $\Tc$ be an $(n-1)$-dimensional CM closed subscheme of $\PP^N$ defined by a homogeneous ideal $J$, and coordinate ring $A=\kk[\T]/J$. Let $I=(h_0,\hdots,h_n)$ be a homogeneous ideal of $A$, with $\deg(h_i)=d$ for all $i$. Let $X:=\Proj(A/I) \subset \Tc$ be finite and locally an almost complete intersection. Set $\nu_0 := \max\{(n-2)d, (n-1)d-\gamma\}$, then $H^0_\mathfrak{m} ( \Sym_A(I) )_\nu =0$ for all $\nu\geq \nu_0$, and $\ann_{\kk[\X]} ( \Sym_A(I)_\nu )\subset \ker(\varphi^\ast)$, for all $\nu \geq \nu_0$.
\end{lem}
\begin{proof}
 The first part follows from \ref{annih}. The proof of the second part can be taken verbatim from \cite[Lemma 2]{BD07}.
\end{proof}

\begin{cor}\label{mainthT}
Let $\Tc$ be an $(n-1)$-dimensional CM closed subscheme of $\PP^N$ defined by a homogeneous ideal $J$, and coordinate ring $A=\kk[\T]/J$. Let $I=(h_0,\hdots,h_n)$ be an homogeneous ideal of $A$, with $\deg(h_i)=d$ for all $i$. Let $X:=\Proj(A/I) \subset \Tc$ be finite and locally almost a complete intersection. Let $\nu_0$ be as above. For any integer $\nu \geq \nu_0$ the determinant $D$ of the complex $(\Zc_\bullet)_\nu$ of $\kk[\X]$-modules defines (up to multiplication with a constant) the same non-zero element in $\kk[\X]$. Moreover, $D=F^{\deg(\varphi)}G$, where $F$ is the implicit equation of $\Sc$.
\end{cor}
\begin{proof}
It follows from Lemma \ref{ZacycALCI}, Lemma \ref{lemAnnih}, and Theorem \ref{locosymalgT}, by following the same lines of the proof of \cite[Thm.\ 5.2]{BuJo03}.
\end{proof}

By \cite[Appendix A]{GKZ94}, the determinant $D$ can be computed either as an alternating product of subdeterminants of the differentials in $(\Zc_\bullet)_{\nu}$ or as the greatest common divisor of the maximal-size minors of the matrix $M$ associated to the
right-most map $(\Zc_1)_{\nu} \rightarrow  (\Zc_0)_{\nu}$ of the $\Zc$-complex (cf.\ Definition \ref{defZcomplex}). Note that this matrix is nothing else than the matrix $M_\nu$ of linear syzygies as described in the introduction; it can be computed with the same algorithm as in \cite{BD07} or \cite{BDD08}. Hence, if $\Tc  \stackrel{\varphi}{\dashrightarrow} \PP^n$ is as in Corollary \ref{mainthT}, the matrix $M_\nu$ of linear syzygies of $h_0,\ldots,h_n$ in degree $\nu \geq \nu_0$ is a representation matrix for the closed image of $\varphi$.

As was done by Bus\'e et al.\ in \cite[Sec.\ 2]{BCJ06}, we conclude that the the extraneous factor $G$ can be described in terms of linear forms.

\begin{prop}\label{extraFactorThm}
If the field $\kk$ is algebraically closed and $X$ is locally generated by at most $n$ elements then, there exist linear forms $L_x\in \kk[\X]$, and integers $e_x$ and $d_x$ such that
\[
 G=\prod_{x\in X} L_x^{e_x-d_x}\in \kk[\X].
\]
Moreover, if we identify $x$ with the prime ideal in $\Spec(A)$ defining the point $x$, $e_x$ is the Hilbert-Samuel multiplicity $e(I_x, A_x)$, and $d_x:= \dim_{A_x/xA_x}(A_x/I_x)$.
\end{prop}

\begin{proof}
 The proof goes along the same lines of \cite[Prop.\ 5]{BCJ06}, just observe that \cite[Lemma 6]{BCJ06} is stated for a Cohen-Macaulay ring as is $A$ in our case.
\end{proof}


\section{The case $\KKK=(\P1)^n$}\label{sec4multiproj}

The aim of this section is computing the implicit equation of a hypersurface in $(\PP^1)^n$. 
 There is also a big spectrum of problems that are not well adapted to taking a common denominator. Typically this process enlarges the base locus of $f$, which could imply having a worse compactification of the domain or an embedding into a bigger projective space. This may also increase considerably the degree of the new maps $f_i$, hence the degree of $f$ which is harmful for the algorithmic approach.

Here we generalize the work in \cite{Bot08}. Hereafter in this section, let $\Tc$ be a $(n-1)$-dimensional projective arithmetically Cohen Macaulay closed scheme over a field $\kk$, embedded in $\PP^N_{\kk}$, for some $N\in \NN$. Write $A=\kk[T_0,\hdots,T_N]/J$ for its CM graded coordinate ring, and let $J$ denote the homogeneous defining ideal of $\Tc$. Set $\T:= T_0,\hdots,T_N$ the variables in $\PP^N$, and $\X$ the sequence $X_1,Y_1,\hdots,X_n,Y_n$, of variables in $(\PP^1)^n$. Write $\mm:=A_+=(\T)\subset A$ for the maximal irrelevant homogeneous ideal of $A$.

Let $\phi$ be a finite map over a relative open set $U$ of $\Tc$ defining a hypersurface in $\PP^n$:
\begin{equation}\label{eqSettingP1n}
 \PP^N \supset \Tc \nto{\phi}{\dto} (\PP^1)^n : \T \mapsto (f_1:g_1)\times\cdots\times(f_n:g_n)(\T),
\end{equation}
where $f_i$ and $g_i$ are homogeneous elements of $A$ of degree $d_i$, for $i=1,\hdots,n$. As in the section before, this map $\phi$ gives rise to a morphism of graded $\kk$-algebras in the opposite sense
\begin{equation}
 \kk[\X] \nto{\phi^\ast}{\lto} A: X_i \mapsto f_i(\T), Y_i \mapsto g_i(\T).
\end{equation}
Since $\ker(\phi^\ast)$ is a principal ideal in $\kk[\X]$, write $H$ for a generator. We proceed as in \cite{Bot08} to get a matrix such that the gcd of its maximal minors equals $H^{\deg(\phi)}$, or a multiple of it.

\medskip

Assume that we are given a rational map like the one in \eqref{initsetting} with $\deg(f_i)=\deg(g_i)=d_i$, $i=1,\hdots,n$. Take $\Tc\subset \PP^N$ the toric embedding obtained from $\Nc'(f)$ cf.\ Definition \ref{defNf}. The multi-projective compactification is given by
\begin{equation}\label{RemToricCaseP1n}
 \AA^n \nto{\iota}{\hto}(\PP^1)^n: (x_1,\hdots,x_n) \mapsto (x_1:1)\times\cdots\times(x_n:1).
\end{equation}
As before, $f$ compactifies via $\rho$ through $\Tc$ to $\phi:\Tc \dto(\PP^1)^n$ as defined in \eqref{eqSettingP1n}. 

\medskip

When $\Tc=\PP^{n-1}$, in \cite{Bot08} was studied the relation between the tensor product of length-one approximation complexes and a Koszul complex $\k.$ (cf.\ Section $3$ loc.\ cit.). In that case, the implicit equation of $\phi$ is computed as a homogeneous resultant by means of a graded strand of $\k.$. We will recall the definition of $\k.$.

We associate to each pair of homogeneous polynomials $f_i$, $g_i$ a linear form $L_i:=Y_i f_i - X_i g_i$ in the ring $R:=A[\X]$ of bidegree $(d,1)$. 
Write $\k.$ for the Koszul complex $\k.(L_1,\hdots,L_n; R)$, associated to the sequence $L_1,\hdots,L_n$ and coefficients in $R$. The $\NN^n$-graded $\kk$-algebra $\BB:= \coker (\bigoplus_i R(-d_i,-1) \to R)$ is the multihomogeneous coordinate ring of the incidence scheme $\Gamma=\overline{\Gamma_\Omega}$. It can be easily observed that $\BB\cong \bigotimes_A \Sym_A(I^{(i)})\cong R/(L_1,\hdots,L_n)$. 

We show that under certain conditions on the $L_i$, there exist an element in $\kk[\X]$ that vanishes whenever $L_1,\hdots,L_n$ have a common root in $\Tc$. This polynomial coincides with the sparse resultant $\Res_{\Tc}(L_1,\hdots,L_n)$. We will see that it is not irreducible in general, in fact, it is not only a power of the implicit equation, it can also have some extraneous factors, while the generic sparse resultant is always irreducible. Those factors come from some components of the base locus of $\phi$ which are not necessarily a common root of all $L_i$: it is enough that one of them vanishes at some point $p$ of $\Tc$ to obtain a base point of $\phi$. We will give sufficient conditions for avoiding extraneous factors.

We compute the implicit equation of the closed image of $\phi$ as a factor of the determinant of $(\k.)_{(\nu,*)}$, for certain degree $\nu$ in the grading of $A$. As in \cite{BDD08} the last map of this complex of vector spaces is a matrix $M_{\nu}$ that represents the closed image of $\phi$. Thus, we focus on the computation of the regularity of $\BB$ in order to bound $\nu$.

\begin{thm} \label{locosymalgT}
Suppose that $A$ is Cohen-Macaulay and $\k.$ is acyclic. Then 
\[
 H^0_\mm(\BB)_\nu=0 \textnormal{ for all }\nu \geq \nu_0 = \paren{\sum_i d_i} -\gamma,
\]
 where $\gamma := \max \{ i \ | \ \mathrm{C_i} \ \mathrm{contains} \ \mathrm{no} \ \mathrm{interior} \ \mathrm{points}    \ \}$. 
\end{thm}
\begin{proof} 
Write $\KK_q$ for the $q$-th object in $\k.$. Consider the two spectral sequences associated to the double complex $C^\bullet_\mm(\k.)$, both converging to the hypercohomology of $\k.$. As $\k.$ is acyclic the first spectral sequence stabilizes at $E_2$-term. The second one has as $E_1$-term $_1''E^p_q =\ H^p_\mm(\KK_q)$. 

Recall that for an $A$-module $M$, $\endd(M)= \max \{ \nu \ | \ M_\nu \neq 0 \}$. The comparison of the two spectral sequences shows that $H^0_\mm(\BB)_\nu$ vanishes as soon as $(_1{''}E^{p}_p)_\nu$ vanishes for all $p$, in fact we have 
\[
 \endd(H^0_\mm(\BB))\leq \max_{p\geq 0}\{\endd(_1{''}E^{p}_p)\}=\max_{p\geq 0}\{\endd(H^p_\mm(\KK_p))\}.
\]
It remains to observe that, since $\KK_p=\bigoplus_{i_1,\hdots,i_p} A(-\sum_{j=1}^p d_{i_j})\otimes_\kk \kk[\X](-p)$ and $\kk[\X]$ is flat over $\kk$, 
\[
 \max_{p\geq 0}\{\endd(H^p_\mm(\KK_p))\}=\max_{p\geq 0} \{\max_{i_1,\hdots,i_p}\{\endd(H^p_\mm(A(-\sum_{j=1}^p d_{i_j})))\}\}.
\]
Hence, we have
\[
 \endd(H^p_\mm(\KK_p))= \left\lbrace\begin{array}{ll}\endd(H^n_\mm(\omega_A^\vee(-\sum_i d_i)))& \mbox{for }p=n, \\
0 & \mbox{otherwise.} \end{array}\right.
\]
Finally, since $(H^n_\mm(\omega_A^\vee))_\nu=0$ for all $\nu\geq -\gamma$, we get 
\[
  \endd(H^n_\mm(\omega_A^\vee(-\sum_i d_i)))=\endd(H^n_\mm(\omega_A^\vee)))+\sum_i d_i<\sum_i d_i -\gamma. \qedhere
\]
\end{proof}

In order to compute the representation matrix $\Xi_\nu$ and the implicit equation of $\phi$, we need to be able to get acyclicity conditions for $\k.$. A theorem due to L.\ Avramov (cf.\ \cite{avr}) gives necessary and sufficient conditions for $(L_1,\hdots,L_n)$ to be a regular sequence in $R$ in terms of the depth of certain ideals of minors of the following matrix $\Xi:=(m_{ij})_{i,j}\in Mat_{2n,n}(A)$
\begin{equation}\label{matrizota}
\Xi=\left(\begin{array}{ccccc}-g_1&0&\cdots&0\\ f_1&0&\cdots&0\\ 
\vdots&\vdots&\ddots&\vdots\\ 0&0&\cdots&-g_n\\ 0&0&\cdots&f_n  \end{array}\right).
\end{equation}
This matrix defines a map of $A$-modules $\psi: A^n \to \bigoplus_{i=0}^n A[x_i,y_i]_1\cong A^{2n}$, we verify that the symmetric algebra $\Sym_A(\coker(\psi))\cong A[\X]/(L_1,\hdots,L_n)$. 
Since $\Sym_A(\coker(\psi))=\BB$ is naturally multigraded, it can be seen as a subscheme of $\Tc\times(\P1)^{n}$. This embedding is determined by the natural projection $A[\X]\to A[\X]/(L_1,\hdots,L_n)$. In fact, the graph of $\phi$ is an $(n-1)$-dimensional irreducible component of $\Proj(\Sym_A(\coker(\psi)))\subset \Tc\times(\P1)^{n}$ which is a projective fiber bundle outside the base locus of $\phi$ in $\Tc$.

\medskip

Let $\Xi$ be a matrix as in (\ref{matrizota}). Write $I_r:=I_r(\Xi)$ for the ideal of $A$ generated by the $r\times r$ minors of $\Xi$, for $0\leq r \leq r_0:=\min\{n+1,m\}$, and define $I_0:=A$ and $I_r:=0$ for $r>r_0$. Theorem \ref{teoRes} below relates both algebraic and geometric aspects. It gives conditions in terms of the ideals of minors $I_r$, for the complex to being acyclic, and on the equation given by the determinant of a graded branch for describing the closed image of $\phi$.
\medskip

\subsection{The implicit equation}

In this part we gather together the facts about acyclicity of the complex $\k.$, and the geometric interpretation of the zeroes of the ideals of minors $I_r$, and we show that under suitable hypotheses no extraneous factor occurs. In order to do this, we introduce some previous notation, following that of \cite{Bot08}. 

Denote by $W$ the closed subscheme of $\Tc \subset \PP^N$ given by the common zeroes of all $2n$ polynomials, write $I^{(i)}$ for the ideal $(f_i,g_i)$ of $A$, and $X$ the base locus of $\phi$, namely
\begin{equation}\label{WX}
 W:=\Proj\paren{A/\sum_{i}I^{(i)}}, \textnormal{ and } \qquad X:=\Proj\paren{A/\prod_{i}I^{(i)}}.
\end{equation}

Set $\alpha\subset [1,n]$, write $I^{(\alpha)}:=\sum_{j\in \alpha} I^{(j)}$, and set $X_{\alpha}:=\Proj(A/I^{(\alpha)})$ and $U_{\alpha}:=X_{\alpha}\setminus\bigcup_{j\notin \alpha}X_{\{j\}}$. If $U_{\alpha}$ is non-empty, consider $p\in U_{\alpha}$, then $\dim(\pi_1^{-1}(p))=|\alpha|$. As the fiber over $U_\alpha$ is equidimensional by construction, write
\begin{equation}\label{VBunbleAlpha}
 \EEE_{\alpha}:=\pi_1^{-1}(U_{\alpha})\subset \Tc\times (\P1)^{n}
\end{equation}
for the fiber over $U_{\alpha}$, which defines a multiprojective bundle of rank $|\alpha|$. Consequently, 
\[
 \codim(\EEE_{\alpha})=n-|\alpha|+(\codim_{\Tc}(U_{\alpha})).
\]
Recall from Definition \ref{defXyOmega} that $\Gamma_\Omega$ is the graph of $\phi$, and set $\Gamma:=\Biproj(\BB)$, the incidence scheme of the linear forms $L_i$. We show in the following theorem that under suitable hypothesis $\Gamma=\overline{\Gamma_\Omega}$, and that $\pi_2(\Gamma)=\Hc$ the implicit equation of the closed image of $\phi$.

\begin{thm}\label{teoRes} Let $\phi: \Tc\dashrightarrow (\P1)^{n}$ be defined by the pairs $(f_i:g_i)$, not both being zero, as in \eqref{eqSettingP1n}. Write for $i=1,\dots,n$, $L_i:=\fxi$ and $\BB:=A[\X]/(L_1,\hdots,L_n)$. Take $\nu_0 = (\sum_i d_i) -\gamma$ as in Theorem \ref{locosymalgT}.
\begin{enumerate}
\item The following statements are equivalent:
 \begin{enumerate}
  \item \label{teo-kos} $\k.$ is a free resolution of $\BB$;
  \item \label{teo-dep} $\codim_A(I_r)\geq n-r+1$ for all $r=1,\hdots,n$;
  \item \label{teo-cod} $\dim(\bigcap_{\alpha\subset [1,n]}( V(\prod_{j\in \alpha}I^{(j)})))\leq |\alpha|-2$. 
 \end{enumerate}
\item\label{teo-equ} If any (all) of the items above are satisfied, then $M_\nu$ has generically maximal rank, namely $\binom{n-1+\nu}{\nu}$. Moreover, if for all $\alpha\subset [1,n]$, $\codim_A(I^{(\alpha)})> |\alpha|$, then, 
\[
 \det((\k.)_\nu)=\det(M_\nu)=H^{\deg(\phi)}, \quad \textnormal{for }\nu\geq\nu_0,
\]
where $\det(M_\nu)$ and $H$ is the irreducible implicit equation of the closed image of $\phi$.
\end{enumerate}
\end{thm}

\begin{proof}
For proving the first part, we refer the reader to \cite{Bot08}, where it has been studied the case $\Tc=\PP^{n-1}$. It remains to observe that the only necessary condition over $A$ is to be Cohen-Macaulay.

\medskip
For proving the second part, the hypotheses have been taken in such a way that $\codim_A(\sum_{j\in \alpha}I^{(j)})> |\alpha|$, for all $\alpha\subset [1,n]$, which implies that $\codim_{\Tc}(U_{\alpha})>|\alpha|$, thus 
\[
\codim(\EEE_{\alpha})>n=\codim(\Gamma_\Omega). 
\]
Set $\Gamma_U:=\coprod_{\alpha}\EEE_{\alpha}$, and observe that $\Gamma\setminus\Gamma_U=\Gamma_\Omega$. Clearly, $\codim(\Gamma_U)>n=\codim(\Gamma_\Omega)=\codim(\overline{\Gamma_\Omega})$.

Since $\Spec (\BB)$ is a complete intersection, in $\AA^{2n}$ it is unmixed and purely of codimension $n$. As a consequence, $\Gamma\neq \emptyset$ is also purely of codimension $n$. This and the fact that $\codim (\Gamma_U)>n$ implies that $\Gamma=\overline{\Gamma_\Omega}$. The graph $\Gamma_\Omega$ is irreducible hence $\Gamma$ as well, and its projection (the closure of the image of $\phi$) is of codimension-one. 

It remains to observe that $\k.$ is acyclic, and $H_0(\k.)\cong\BB$. Considering the homogeneous strand of degree $\nu>\eta$ we get the following chain of identities (cf.\ \cite{KMun}):
\[
 \begin{array}{rl}
[\det((\k.)_\nu)]&=\div_{\kk[\X]}(H_0(\k.)_\nu)\\
&=\div_{\kk[\X]}(\BB_\nu)\\
&=\sum_{\pp \textnormal{ prime, }\codim_{\kk[\X]}(\pp)=1}\length_{\kk[\X]_\pp} ((\BB_\nu)_\pp)[\pp]. 
\end{array}
\]
Our hypothesis were taken in such a way that only one prime occurs. Also since 
\[
 [\det((\k.)_\nu)]=\div_{\kk[\X]}(\res)= e \cdot [\qq],
\]
for some integer $e$ and $\qq:=(H)\subset \kk[\X]$, we have that 
\[
 \sum_{\pp \textnormal{ prime, }\codim(\pp)=1}\length_{\kk[\X]_\pp} ((\BB_\nu)_\pp)[\pp]=e\cdot [\qq],
\]
and so $[\det((\k.)_\nu)]=\length_{\kk[\X]_\qq} ((\BB_\nu)_\qq)[\qq]$. As the $\Gamma_\Omega$ is irreducible, $\Gamma$ is also irreducible and we have
\[
 \length_{\kk[\X]_\qq} ((\BB_\nu)_\qq)=\dim_{\kkk(\qq)}{(\BB_\nu\otimes_{\kk[\X]_\qq} \kkk(\qq))}=\deg(\phi),
\]
where $\kkk(\qq):=\kk[\X]_\qq/\qq \cdot \kk[\X]_\qq$ which completes the proof.
\end{proof}

\subsection{Analysis of the extraneous factors}

Theorem \ref{teoRes} can be generalized (in the sense of \cite[Sec.\ 4.2]{Bot08}) taking into account the fibers in $\Tc\times (\PP^1)^n$ that give rise to extraneous factors, by relaxing the conditions on the ideals $I_r$ stated in Theorem \ref{teoRes}. Recall from \eqref{WX} that $W:=\Proj(A/\sum_{i}I^{(i)})$ and $X:=\Proj(A/\prod_{i}I^{(i)})$, and that for each $\alpha\subset [1,n]$, $I^{(\alpha)}:=\sum_{j\in \alpha} I^{(j)}$, $X_{\alpha}:=\Proj(A/I^{(\alpha)})$ and $U_{\alpha}:=X_{\alpha}\setminus\bigcup_{j\notin \alpha}X_{\{j\}}$. As was defined in \eqref{VBunbleAlpha}, $ \EEE_{\alpha}:=\pi_1^{-1}(U_{\alpha})\subset \Tc\times (\P1)^{n}$ is a multiprojective bundle of rank $|\alpha|$ over $U_{\alpha}$, such that $\codim(\EEE_{\alpha})=n-|\alpha|+(\codim_{\Tc}(U_{\alpha}))$.

\begin{lem}\label{lemCodimAvramov}
 Let $\phi: \Tc \dashrightarrow (\P1)^{n}$, be a rational map satisfying condition 1 in Theorem \ref{teoRes}. Then, for all $\alpha\subset [1,n]$, $\codim_A(I^{(\alpha)})\geq |\alpha|$.
\end{lem}

\begin{proof} Same proof of \cite{Bot08}, it is only needed that $A$ be Cohen-Macaulay.
\end{proof}

We proceed as in loc.\ cit.\, defining the basic language needed to describe the geometry of the base locus of $\phi$.

\begin{defn}\label{defUOmegaXalpha}
For each $\alpha\subset [1,n]$, denote by $\Theta:=\{\alpha\subset [1,n]\ :\ \codim(I^{(\alpha)})=|\alpha|\}$. Hence, let $I^{(\alpha)}=(\cap_{\qq_i\in \Lambda_\alpha} \qq_i)\cap \qq'$ be a primary decomposition, where $\Lambda_\alpha$ is the set of primary ideals of codimension $|\alpha|$, and $\codim_A(\qq')>|\alpha|$. Write $X_{\alpha,i}:=\Proj(A/\qq_i)$ with $\qq_i\in \Lambda_\alpha$, and let $X_{\alpha,i}^{red}$ be the associated reduced variety. 

Write $\alpha:=\{i_1,\hdots,i_k\}\subset [1,n]$, and denote by $\pi_\alpha: (\P1)^{n}\to (\P1)^{n-|\alpha|}$ the projection given by 
\[
 \pi_\alpha: (x_1:y_1)\times \cdots \times(x_n:y_n)\mapsto (x_{i_{k+1}}:y_{i_{k+1}})\times \cdots \times(x_{i_{n}}:y_{i_{n}}).
\]

Define $\PP_\alpha:=\pi_\alpha((\P1)^{n})\cong (\P1)^{n-|\alpha|}$, and $\phi_{\alpha}:=\pi_{\alpha}\circ \phi:\Tc\dto \PP_\alpha$. 

 Denote by $W_\alpha$ the base locus of $\phi_{\alpha}$. Clearly $W\subset W_{\alpha}\subset X$ (cf.\ equation \eqref{WX}). Denote $\UU_{\alpha}:=\Tc\setminus W_{\alpha}$, the open set where $\phi_{\alpha}$ is well defined. Write $\Omega_{\alpha}:= X_{\alpha}\cap \UU_{\alpha}$ and $\Omega_{\alpha,i}:= X_{\alpha,i}\cap \UU_{\alpha}$. If $\alpha$ is empty, we set $\pi_\alpha=Id_{(\P1)^{n}}$, $\phi_\alpha=\phi$, $W_\alpha=W$ and $\UU_{\alpha}=\Omega_{\alpha}=\Omega$.

We get a commutative diagram as follows
\[
 \xymatrix@1{
 \Omega_\alpha\ \ar@{^(->}[r]\ar[rrrd]_{\phi_{\alpha}|_{\Omega_\alpha}}& \ X_\alpha\ \ar@{^(->}[r]& \ \Tc\ \ar@{-->}[r]^{\phi}\ar@{-->}[rd]^{\phi_{\alpha}}& \ (\P1)^{n}\ar@{->>}^{\pi_{\alpha}}[d]&\\
 &&&  **[r]\ \PP_\alpha. }
\]
\end{defn}

\begin{rem}
 Let $p\in \Tc$ be a point, then there exist a unique pair $(\alpha,i)$ such that $p\in \Omega_{\alpha,i}$. If $p\in W$, then $\alpha=\emptyset$ and no $i$ is considered.
\end{rem}
\begin{proof}
 It is clear by definition of $\Omega_\alpha$ that if $p\in W$, then $\alpha=\emptyset$ and no $i$ needs to be considered. Hence, assume that $p\in \Tc\setminus W$. Thus, we define $\alpha:=\{i\in [1,n]\ :\ f_i(p)=g_i(p)= 0\}$ which is a non-empty subset of $[1,n]$. For this set $\alpha$, define $\phi_\alpha$ according to Definition \ref{defUOmegaXalpha}, set $W_\alpha$ the base locus of $\phi_\alpha$ and $X_\alpha:=\Proj(A/I^{(\alpha)})$. By definition, $p\in \Omega_\alpha:=X_\alpha\setminus W_\alpha$. Since, in particular, $p\in X_\alpha$, it is one of its irreducible components that we denote by $X_{\alpha,i}$ following the notation of Definition \ref{defUOmegaXalpha}. We conclude that $p\in \Omega_{\alpha,i}:=X_{\alpha,i}\setminus W_\alpha$, from which we obtain the $(\alpha,i)$ of the statement.
\end{proof}

In the following lemma we define a multiprojective bundle of rank $|\alpha|$ over $\Omega_{\alpha,i}$.

\begin{lem}\label{bundle} 
 For $\phi$ as in Theorem \ref{teoRes}, and for each $\alpha \in \Theta$ and each $\qq_i\in \Lambda_\alpha$, the following statements are satisfied:
\begin{enumerate}
 \item $\Omega_{\alpha,i}$ is non-empty
 \item for all $p\in \Omega_{\alpha,i}$, $\dim(\pi_1^{-1}(p))=|\alpha|$
 \item the restriction $ \phi_{\alpha,i}$ of $\phi$ to $\Omega_{\alpha,i}$, defines a rational map 
\begin{equation}\label{phiai}
 \phi_{\alpha,i}:X_{\alpha,i}\dto \PP_\alpha\cong (\P1)^{n-|\alpha|}.
\end{equation}
\item $Z_{\alpha,i}:=\pi_1^{-1}(\Omega_{\alpha,i})\nto{\pi_1}{\lto}\Omega_{\alpha,i}$ defines a multiprojective bundle $\EEE_{\alpha,i}$ of rank $|\alpha|$ over $\Omega_{\alpha,i}$.
\end{enumerate}
\end{lem}

\begin{proof}
Fix $X_{\alpha,i}\subset X_{\alpha}$ and write $\alpha:=i_1,\hdots,i_k$. As $\Omega_{\alpha,i}=X_{\alpha,i}\setminus \bigcup_{j\notin \alpha}X_{\{j\}}$ it is an open subset of $X_{\alpha,i}$. If $\Omega_{\alpha,i}=\emptyset$ then $X_{\alpha,i}\subset \bigcup_{j\notin \alpha}X_{\{j\}}$, and as it is irreducible, there exists $j$ such that $X_{\alpha,i}\subset X_{\{j\}}$, hence $X_{\alpha,i}\subset X_{\{j\}}\cap X_\alpha= X_{\alpha\cup \{j\}}$. Denote by $\alpha':= \alpha\cup \{j\}$, it follows that $\dim(X_{\alpha'})\geq \dim(X_{\alpha,i})=n-|\alpha|>n-|\alpha'|$, which contradicts the hypothesis.

Let $p\in \Omega_{\alpha,i}$, $\pi_1^{-1}(p) = \{p\}\times\{q_{i_{k+1}}\}\times\cdots\times\{q_{i_{n}}\}\times(\P1)^{|\alpha|}$, where the point $q_{i_{j}}\in \P1$ is the only solution to the nontrivial equation $L_{i_{j}}(p,x_{i_{j}},y_{i_{j}})=y_{i_{j}}f_{i_{j}}(p)-x_{i_{j}}g_{i_{j}}(p)=0$. Then we deduce that $\dim(\pi_1^{-1}(p))=|\alpha|$, and that $\phi_{\alpha,i}:\Omega_{\alpha,i}\to \PP_\alpha$ given by $p\in \Omega_{\alpha,i}\mapsto \{q_{i_{k+1}}\}\times\cdots\times\{q_{i_{n}}\}\in \PP_\alpha$, is well defined.

The last statement follows immediately from the previous ones. 
\end{proof}

We get the following result.

\begin{thm}\label{teoResGral}  Let $\phi: \Tc\dashrightarrow (\P1)^{n}$ be defined by the pairs $(f_i:g_i)$, not both being zero, as in \eqref{eqSettingP1n}. Assume that $\codim_A(I_r)\geq n-r+1$ for all $r=1,\hdots,n$. Denote by $H$ the irreducible implicit equation of the closure of its image. Then, there exist relative open subsets, $\Omega_{\alpha,i}$, of $\Tc$ such that the restriction $\phi_{\alpha,i}$ of $\phi$ to $\Omega_{\alpha,i}$ defines a rational map $\phi_{\alpha,i}:\Omega_{\alpha,i}\to \PP_\alpha$ and positive integers $\mu_{\alpha,i}$ such that:
\[
 \res_{\Tc}(L_0,\hdots,L_n)=H^{\deg(\phi)}\cdot \prod_{\alpha,i} (H_{\alpha,i})^{\mu_{\alpha,i}\cdot\deg(\phi_{\alpha,i})}.
\]
\end{thm}
\begin{proof}
 The proof of this result follows similar lines of that of \cite[Thm.\ 22]{Bot08}. Recall $\Gamma:=\Biproj(\BB)$, and set $\Gamma_0:=\overline{\Gamma_\Omega}$, the closure of the graph of $\phi$. Applying $\pi_2$ to the decomposition $\Gamma\setminus \Gamma_U=\Gamma_0$ we see that $[\pi_2(\Gamma_U)]=[\res_{\Tc}(L_0,\dots,L_n)]-[\pi_2(\Gamma_0)]$ is the divisor associated to the extraneous factors. It is clear that $[\pi_2(\Gamma_U)]$ defines a principal divisor in $(\P1)^{n}$ denote by  $G=\frac{\res_{\Tc}(L_0,\dots,L_n)}{H^{\deg(\phi)}}$, with support on $\pi_2(\Gamma\setminus \Gamma_0)$, and that $\Gamma$ and $\Gamma_0$ coincide outside $X\times (\P1)^{n}$. 

By Lemma \ref{bundle}, for each $\alpha$ and each $\qq_i\in \Delta_\alpha\subset\Lambda_\alpha$, $\phi_{\alpha,i}$ defines a multiprojective bundle $\EEE_{\alpha,i}$ of rank $|\alpha|$ over $\Omega_{\alpha,i}$. 

By definition of $\Delta_\alpha$, $\overline{\pi_2(\EEE_{\alpha,i})}$ is a closed subscheme of $(\P1)^{n}$ of codimension-one. Denoting by $[\overline{\EEE_{\alpha,i}}]={\mu_{\alpha,i}}\cdot[\overline{\EEE_{\alpha,i}^{red}}]$ the class of $\overline{\EEE_{\alpha,i}}$ as an algebraic cycle of codimension $n$ in $\PP^{n-1}\times(\P1)^n$, we have $(\pi_2)_*[\overline{\EEE_{\alpha,i}}]= {\mu_{\alpha,i}}\cdot(\pi_2)_*[\overline{\EEE_{\alpha,i}^{\ red}}]= {\mu_{\alpha,i}}\cdot\deg(\phi_{\alpha,i})\cdot[\pp_{\alpha,i}]$, where $\pp_{\alpha,i}:=(H_{\alpha,i})$. 

As in Theorem \ref{teoRes}, one has for $\nu>\eta$:
\[
 \begin{array}{rl}
[\det((\k.)_\nu)]&=\div_{k[\X]}(H_0(\k.)_\nu)\\
&=\div_{k[\X]}(\BB_\nu)\\
&=\sum_{\pp \textnormal{ prime, }\codim_{k[\X]}(\pp)=1}\length_{k[\X]_\pp} ((\BB_\nu)_\pp)[\pp].
\end{array}
\]
We obtain that
\[
 [\det((\k.)_\nu)]= \sum_{\alpha\in \Theta} \sum_{\pp_{\alpha,i}} \length_{k[\X]_{\pp_{\alpha,i}}} ((\BB_\nu)_{\pp_{\alpha,i}})[\pp_{\alpha,i}]+\length_{k[\X]_{(H)}} ((\BB_\nu)_{(H)})[(H)].
\]

In the formula above, for each $\pp_{\alpha,i}$ we have
\[ 
\length_{k[\X]_{\pp_{\alpha,i}}} ((\BB_\nu)_{\pp_{\alpha,i}})=\dim_{\kk(\pp_{\alpha,i})}{(\BB_\nu\otimes_{k[\X]_{\pp_{\alpha,i}}} \kk(\pp_{\alpha,i}))}={\mu_{\alpha,i}}\cdot\deg(\phi_{\alpha,i}),
\]
where $\kk(\pp_{\alpha,i}):=k[\X]_{\pp_{\alpha,i}}/\pp_{\alpha,i}\cdot k[\X]_{\pp_{\alpha,i}}$.

Consequently we get that for each $\alpha\in \Theta$, there is a factor of $G$, denoted by $H_{\alpha,i}$, that corresponds to the irreducible implicit equation of the scheme theoretic image of $\phi_{\alpha,i}$, raised to a certain power $\mu_{\alpha,i}\cdot\deg(\phi_{\alpha,i})$.
\end{proof} 

\begin{rem}\label{rem-notHypersurface}
 Observe that if $\im(\phi_{\alpha,i})$ is not a hypersurface in $\PP_\alpha$ then $\deg(\phi_{\alpha,i})$ is $0$, hence $(H_{\alpha,i})^{\mu_{\alpha,i}\cdot\deg(\phi_{\alpha,i})}=1$. Thus $\phi_{\alpha,i}$ does not give an extraneous factor.
\end{rem}


\section{The algorithmic approach: Hilbert and Ehrhart functions}\label{sec5algorithmic}

In this section we focus on the study of the size of the matrices $M_\nu$ obtained in the two cases $\KKK=\PP^n$ and $\KKK=(\PP^1)^n$. Let us analyze first the case $\KKK=\PP^n$, thus, where we get a map $\varphi:\Tc \dto \PP^n$ as defined in \eqref{eqSettingPn}. Assume also that the base locus of $\varphi$ is a zero-dimensional almost locally complete intersection scheme. Hence, the associated $\Zc$-complex is acyclic. We have shown in Section \ref{sec3Pn} that the matrix $M_\nu$ is obtained as the right-most map of the $(\nu,*)$-graded strand of the approximation complex of cycles $\Zc_\bullet(\h,A)_{(\nu,*)}$:
\begin{equation*}
\ 0 \to (\Zc_n)_{(\nu,*)}(-n) \to (\Zc_{n-1})_{(\nu,*)}(-(n-1)) \to \cdots \to (\Zc_1)_{(\nu,*)}(-1) \nto{M_\nu}{\to} (\Zc_0)_{(\nu,*)} .
\end{equation*}
Write $h_B(\mu):=\dim_\kk(B_\mu)$ for the Hilbert function of $B$ at $\mu$. Since $\Zc_i= Z_i[i \cdot d] \otimes_A A[\X]=Z_i[i \cdot d] \otimes_\kk \kk[\X]$, $(\Zc_i)_{(\nu,*)}=(Z_i[i \cdot d])_\nu \otimes_\kk \kk[\X]$, we have $M_\nu \in \mat_{h_{A}(\nu),h_{Z_1}(\nu+d)}(\kk[\X])$.

\medskip

Let $\KKK=(\P1)^n$, and assume we are given a map $\phi: \Tc \dto (\P1)^n$ as the one considered in \eqref{eqSettingP1n}, satisfying the conditions of Theorem \ref{teoResGral}. We obtain the matrix $M_\eta$ computed from the Koszul complex $(\k.)_{(\eta,*)}$. Hence, the matrix $M_\eta$ belongs to $\mat_{h_A(\eta),n h_A(\eta-d)}(\kk[\X])$.

Both numbers $h_{A}(\nu)$ and $h_{Z_1}(\nu+d)$, in the projective and multiprojective setting, can be computed easily in Macaulay2. The cost of computation depends on the ring structure of $A$. When $A$ is just any finitely generated $\NN$-graded Cohen-Macaulay $\kk$-algebra, finding a precise theoretical estimate of these numbers would be very difficult. Also, the module structure of $Z_1$ can also be very intricate. Since it is a $\NN$-graded sub-$A$-module of $A^{n+1}$, we have $h_{Z_1}(\nu+d)\leq (n+1)h_{A}(\nu+d)$.

\medskip

Assume now that the ring $A$ is the coordinate ring of a normal toric variety $\Tc$ defined from a polytope $\Nc$, as mentioned in Section \ref{ImageCodim1}, and later in Remarks \ref{RemToricCasePn} and \eqref{RemToricCaseP1n}. In this setting, the situation above can be rephrased in a more combinatorial fashion. Let $\Nc$ be a $(n-1)$-dimensional lattice polytope, that is a full-dimensional convex polytope in $\RR^{n-1}$ with vertices lying in $\ZZ^{n-1}$ . For any integer $k \geq 0$, the multiple $k\Nc=\{p_1+\cdots+p_k\ :\ p_i\in \Nc\}$ is also a lattice polytope, and we can count its lattice points. The function taking each integer $k\in \NN$ to the number $E_\Nc(k) = \#(k\Nc) \cap \ZZ^{n-1}$ of lattice points in the polytope $k\Nc$ is the \textit{Ehrhart function} of $\Nc$ (cf.\ \cite{MS}). Write $E^+_\Nc(k)= \#\relint(k\Nc) \cap \ZZ^{n-1}$, the number of integer points in the interior of $k\Nc$ (cf.\ \cite{Latte} for a software for computing those numbers). It is known that there is an identification between $\kk[\relint(C)]$ and $\omega_A$, hence, this can be understood as $E^+_\Nc(k)= h_{\omega_A}(k)$.

Let $C$ be the cone in $\RR^{n-1} \times \RR$ spanned in degree $1$ by the lattice points in the polytope $\Nc$, which is normal by assumption, hence $A$ is Cohen-Macaulay (cf.\ \ref{NfAlwaysNormal}). Assume $\Nc'$ stands for some integer contraction of $\Nc$ which is also normal and take $d\in \NN$ such that $d\Nc'=\Nc$. Then $A'=\kk[\Nc']$ its Cohen-Macaulay semigroup ring. As $d\Nc'=\Nc$, we have that $E_{\Nc'}(d \mu)=E_{\Nc}(\mu)$ for all $\mu$. Set $\gamma:= a_n(A)=\inf\{\mu\ :\ (\omega_A^\vee)_\mu=0\}$ and $\gamma':= a_n(A')=\inf\{\mu\ :\ (\omega_{A'}^\vee)_\mu=0\}$. As $(\omega_A^\vee)_{\mu} = \Hom_\kk(M_{-\mu},\kk)$, we have that $\gamma= \max \{ i \ : \ C_i \ \textnormal{contains } \textnormal{no } \textnormal{interior } \textnormal{points}\}$, where $C_i:=C\cap \ZZ^{n-1}\times \{i\}$, and similarly for $\gamma'$. For a deeper understanding we refer the reader to \cite[Sec.\ 5]{BH}.

Both $A$ and $A'$ give rise to two different -but related- implicitization problems, the following result gives a condition on the rings $A$ and $A'$ to decide when it is algorithmically better to choose one situation or the other.
 
\begin{lem} Take $\Nc$, $\Nc'$, $d$, $\gamma$ and $\gamma'$ as above. Then
 \begin{enumerate}
  \item $\gamma\geq \gamma'$;
  \item $d (\gamma'+1)\geq \gamma+1$;
 \end{enumerate}
\end{lem}
\begin{proof}
 As $d\geq 1$, we can assume $\Nc'\subset \Nc$, hence, the first item follows. For the second item, we just need to observe that if $\mu \Nc \cap \ZZ^n$ is nonempty, then $\mu d \Nc' \cap \ZZ^n$ neither it is. Taking $\mu$ the smallest positive integer with this property, and writing $\gamma=\mu+1$, the second item follows.
\end{proof}

\begin{rem}
 Is not true in general that $d (\gamma+1)> \gamma'+1$: take $\Nc$ as the triangle with vertices $(3,0)$, $(0,3)$, $(0,0)$ and $\Nc'$ the triangle with vertices $(1,0)$, $(0,1)$, $(0,0)$; hence $d=3$, $\gamma=0, \gamma'=2$. We obtain $d (\gamma +1)=3=\gamma' +1$, which shows also that $d \gamma$ need not be bigger than $\gamma'$. It is neither true that $d(\gamma+1)= \gamma'+1$, for instance, take $\Nc$ as the triangle with vertices $(4,0)$, $(0,4)$, $(0,0)$ and $\Nc'$ as before. Observe that $d (\gamma +1)=4(0+1)=4>\gamma' +1=2+1=3$.
\end{rem}
\medskip

\begin{lem}
 Let $f_0,\hdots,f_n$ be with generic coefficients and same denominator, $\Nc:=\Nc(f_0,\hdots,f_n)$, $\Nc'$ and $d$ such that $d \Nc'=\Nc$. Take $\nu_0=(n-1)-\gamma$ the bound established in \cite[Thm.\ 11]{BDD08}, and $\nu_0'=d(n-1)-\gamma'$. Write $\delta :=d (\gamma+1)- (\gamma'+1)$. Then $E_\Nc(\nu_0)>E_{\Nc'}(\nu_0')$ if and only if $\delta>d-1$. 
\end{lem}
\begin{proof}
 We have seen that $E_{\Nc'}(d \nu_0)=E_{\Nc}(\nu_0)$, hence, it is enough to compare $E_{\Nc'}(d \nu_0)$ and $E_{\Nc'}(\nu_0')$. Writing $d\gamma= \gamma'+\delta-(d-1)$, we have
\begin{equation*}
 E_{\Nc'}(d \nu_0)=E_{\Nc'}(d(n-1)-d\gamma)=E_{\Nc'}(d(n-1)-\gamma'+\delta-(d-1)),
\end{equation*}
from where we deduce that $E_\Nc(\nu_0)>E_{\Nc'}(\nu_0')$ if and only if $\delta>d-1$.
\end{proof}

\medskip

\begin{cor}
 Let $f:\AA^{n-1}\dto \AA^n$ be a rational map with generic coefficients and same denominator, $\Nc:=\Nc(f)$, $\Nc'$ normal polytopes and $d$ such that $d \Nc'=\Nc$. Let $\Tc$ and $\Tc'$ be the arithmetically Cohen-Macaulay toric varieties defined from $\Nc$ and $\Nc'$ respectively, and $\varphi:\Tc\subset \PP^{E_\Nc(1)}\dto \PP^n$ and $\varphi':\Tc'\subset \PP^{E_{\Nc'}(1)}\dto \PP^n$.  Take $\nu_0$, $\nu_0'$ and $\delta$ as above. And write $M_{\nu_0}$ and $M'_{\nu_0'}$ the representation matrices of $\im(\varphi)$ and $\im(\varphi')$ respectively. Then $\#\textnormal{rows}(M_{\nu_0})>\#\textnormal{rows}(M'_{\nu_0'})$ if and only if $\delta>d-1$.
 
\end{cor}

\medskip
In the second case, given a map $\phi:\Tc \dto (\P1)^n$ as in Theorem \ref{teoResGral}, we obtain the matrix $M_\nu$ as the right-most matrix from the Koszul complex $(\k.)_{(\nu,*)}:$
\[
 \ 0 \to A_{\nu-n d}\otimes_\kk \kk[\X](-n) \to \cdots \to (A_{\nu-d})^n\otimes_\kk \kk[\X](-1)  \nto{M_\nu}{\lto} A_\nu\otimes_\kk \kk[\X] \to 0, 
\]
It is clear that $M_\nu$ is a $\dim_\kk(A_{\nu})$ by $\dim_\kk((A_{\nu-d})^n)$ matrix. As  $\bigoplus_{k\geq 0}\gen{C_k}_\kk=\kk[C]$ which is canonically isomorphic to $A$, and also $\dim_\kk(A_{\nu})=E_\Nc(\nu)$ and $\dim_\kk((A_{\nu-d})^n)=n E_\Nc(\nu-d)$, hence
\begin{equation}
 M_\nu \in \mat_{E_\Nc(\nu),n E_\Nc(\nu-d)}(\kk[\X]).
\end{equation}


\section{Examples}\label{sec6examples}

In this section we show, in a few examples, how the theory developed in earlier sections works. We first analyze two concrete examples of parametrized surfaces, given as the image of a rational map defined by rational functions with different denominators. There, we show the advantage of not taking a common denominator, and hence, we regard their images in $(\P1)^3$. Later we show how the method is well adapted for generic rational affine maps with a fixed polytope.

In the last part we show for the setting where $\KKK=(\P1)^n$, for small $n$, how the splitting of the base locus is done in order to obtain the family of multiprojective bundles.

\medskip

\begin{exmp}
 We consider here an example of a very sparse parametrization where the multihomogeneous compactification of the codomain is fairly better than the homogeneous compactification. Take $n=3$, and consider the affine map

\[
 f: \AA^2 \dto \AA^3: (s,t) \mapsto \paren{\frac{st^6+2}{st^5-3st^3},\frac{st^6+3}{st^4+5s^2t^6},\frac{st^6+4}{2+s^2t^6}}.
\]
 Observe that in this case there is no smallest multiple of the Newton polytope $\Nc(f)$ with integer vertices, hence, $\Nc(f)=\Nc'(f)$ as can be seen in the picture below.
\begin{center}
 \includegraphics[scale=0.6]{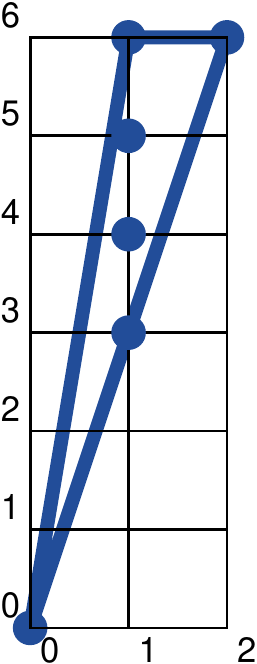}\label{ex2BDD}
\end{center}

Computing in Macaulay2 we get that the homogeneous coordinate ring is 
\[
 A=\frac{k[T_0,\ldots,T_5]}{(T_3^2-T_2T_4,T_2T_3-T_1T_4, T_2^2-T_1T_3,T_1^2-T_0T_5)}.
\]
When $\AA^3$ is compactified into $\PP^3$ we obtain from $f$ a new map $\varphi: \Tc\dto \PP^3$ by replacing $(s,t)$ by $T_0,\ldots,T_5$, and taking a common donominator. We can easily see that taking common denominator leads to polynomials of degree up to $23$ and the Newton polytope of the four new polynomials contains $26$ integer points instead of $6$. Again computing in Macaulay2, for $\nu_0=2$, the matrix $M_\nu$ has $351$ rows and about $500$ columns. It can be verified that this compactification gives a base point which is not locally a complete intersection, but locally an almost complete intersection, giving rise to extraneous factors.

On the other hand, compactifying $\AA^3$ into $(\P1)^3$ we get the map
{\small\[
\begin{array}{rcl}
 \phi: \Tc &\dashrightarrow & \PP^1\times \PP^1\times \PP^1\\
(T_0,\ldots,T_5) & \mapsto & (2T_0+T_4:-3T_1+T_3)(3T_0+T_4:T_2+5T_5)(4T_0+T_4:2T_0+T_5)
\end{array}
\]}

Computations in Macaulay2 give that for $\nu_0=3$ the matrix $M_{\nu_0}$ is of size $34 \times 51$. Since there are no base points with two-dimensional fibers, we get no extraneous factors and hence, $H^{\deg (\phi)}$ can be computed as $\frac{(34\times 34)\textnormal{-matrix}\cdot (1\times 1)\textnormal{-matrix}}{(17\times 17)\textnormal{-matrix}}$, getting an equation of degree (6,6,6).
\end{exmp}

\medskip

\begin{exmp}
 This example shows how the methods work in the generic case with a fixed polytope. We begin by taking $\Nc$ a normal lattice polytope in $\RR^{n-1}$. For the sake of clarity we will treat a particular case in small dimension. Hence, set $n=3$, and consider $\Nc$ as in the drawing below. It will remain clear that this example can be generalized to any dimension and any polytope with integer vertices. 

\begin{center}
 \includegraphics[width=0.15\textwidth]{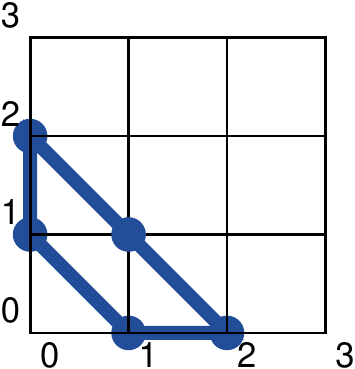}\label{ejemplo3}
\end{center}

Assume we are given six generic polynomials $f_1,f_2,f_3,g_1,g_2,g_3$ with support in $\Nc$, hence we get an affine rational map $f:\AA^2\dto \AA^3$ given by $(s,t)\mapsto(\frac{f_1}{g_1},\frac{f_2}{g_2},\frac{f_3}{g_3})$. We write $f_i=\sum_{(a,b)\in \Nc} U_{(a,b),i}\cdot s^at^b$, and $g_i=\sum_{(a,b)\in \Nc} V_{(a,b),i}\cdot s^at^b$. Set $\UU:=\{U_{(a,b),i}, V_{(a,b),i}:\ \textnormal{ for all }(a,b)\in \Nc, \textnormal{ and } i=1,2,3\}$, the set of coefficients, and define $\kk:=\ZZ[\UU]$.

\medskip

Now we focus on computing the implicit equation of a convenient compactification for the map. Let $\Tc$ be the toric variety associated to the Newton polytope $\Nc$, embedded in $\PP^4$. We will compare how the method works in the $\PP^3$ and $(\PP^1)^3$ compactifications of $\AA^3$ with domain $\Tc$. One key point to remark is that these two maps have no base points, since we are taking the toric compactification associated to $\Nc$ and generic coefficients, hence, we will not have any extraneous factors.

In the first case, we take common denominator obtaining four polynomials with generic coefficients in the polytope $3 \Nc$. If we consider the smallest multiple, we recover the polytope $\Nc$, and maps of degree $3$. We obtain in this case that $f$ factorizes through $\Tc\subset \PP^4$  via $\varphi:\Tc \dto \PP^3$ given by $4$ polynomials of degree $3$ in the variables $T_0,\hdots,T_4$. From Lemma \ref{lemAnnih}, we take $\nu_0 := \max\{3, 6-\gamma\}$. Since $2 \Nc$ has integer interior points but $\Nc$ does not, $\gamma=1$, thus $\nu_0=5$. Now, since $X$ is empty in $\Tc$, from Lemma \ref{ZacycALCI}, the complex $\Zc_\bullet$ is acyclic.

From Theorem \ref{mainthT} we see that the implicit equation can be computed as the determinant of the complex $(\Zc_\bullet)_{\nu}$ for $\nu\geq \nu_0$, or as the gcd of the maximal minors of the right-most map $(\Zc_1)_5(-1) \nto{M_\nu}{\lto} (\Zc_0)_5$. We can easily compute the dimension of $A_5$, by the formula $\#(k\cdot \Nc)= (k+1)(k+1+k/2)$. When $k=5$, we get $\#(5\cdot \Nc)=51$, hence $(\Zc_0)_5 = \kk^{51}[\X]$. Since $M_\nu$ gives a surjective map, $(\Zc_1)_5(-1)$ has dimension bigger than or equal to $51$.

\medskip

Instead of taking common denominator, we can proceed by compactifying $\AA^3$ into $(\P1)^3$. In this case we get a map $\phi:\Tc \dto (\P1)^3$ is given by $3$ pairs of linear functions on the variables $T_0,\hdots,T_4$. 

From Theorem \ref{locosymalgT}, we take $\nu \geq \nu_0 = 1+1+1-1=2$. Now, since the polynomials $f_i$ and $g_i$ have generic coefficients, hence $L_i:=Y_i f_i - X_i g_i$ does as well, thus $\k.:=\k.(L_1,L_2,L_3;A[\X])$ is acyclic. From Lemma \ref{teoRes}, the implicit equation can be computed as the determinant of the complex $(\k.)_{\nu}$ for $\nu\geq 2$, or as the gcd of the maximal minors of the right-most map $A_1^{3}[\X](-1)  \nto{M_2}{\lto} A_2[\X]$. Since $\dim(A_0)=1$, $\dim(A_1)=5$ and $\dim(A_2)=12$ we get the complex $\kk^{3}[\X](-2)\to \kk^{15}[\X](-1)  \nto{M_2}{\lto} \kk^{12}[\X]$. Thus, the implicit equation can be computed as the gcd of the maximal minors of a $(12\times 14)$-matrix, or as $\frac{\det(12\times12\textnormal{-matrix})}{\det(3\times 3\textnormal{-matrix})}$.
\end{exmp}

\medskip

\begin{exmp}
 Assume we are given four tuples of polynomials $f_i$, $g_i$, for $i\in [1,4]$, in three variables $s,t,u$. Let them be $f_1=s+tu^2$, $g_1=u^2$, $f_2=st$, $g_2=u^2$, $f_3=su^2$, $g_3=t$, $f_4=stu^2$, $g_4=1$. They define a rational map $f: \AA^3\dto \AA^4$ given by $(s,t,u)\mapsto (f_1/g_1,f_2/g_2,f_3/g_3,f_4/g_4)$.

In the spirit of this article, we compactify $\AA^3$ into the toric variety associated to the smallest multiple of the Newton polytope the input polynomials define. It is easy to see that this polytope $\Nc$ is a $(1\times 1\times 2)$-parallelepiped, and $\Tc\cong (\P1)^3\subset \PP^{11}$.

In order to detect the extraneous factor that occurs, consider the rational map
\[
 \begin{array}{rcl}
  \tilde{\phi}:(\P1)^3&\dto& (\P1)^4\\
(s:s')\times(t:t')\times(u:u')&\mapsto &(\tilde f_1: \tilde g_1)\times(\tilde f_2:\tilde g_2)\times (\tilde f_3:\tilde g_3)\times (\tilde f_4:\tilde g_4),
 \end{array}
\]
where $( \tilde - )$ means homogenizing with respect to the degree $(1,1,2)$ with the variables $s'$, $t'$ and $u'$.

We easily observe that the base locus has codimension $2$, in fact many lines occur in the base locus: There are 
\begin{enumerate}
 \item four lines $\LL_{1}=(1:0)\times(t:t')\times(1:0)$, $\LL_{2}=(1:0)\times(t:t')\times(0:1)$, $\LL_{3}=(0:1)\times(t:t')\times(1:0)$, $\LL_{4}=(0:1)\times(t:t')\times(0:1)$; 
 \item three lines $\LL_{5}=(1:0)\times(1:0)\times(u:u')$, $\LL_{6}=(1:0)\times(0:1)\times(u:u')$, $\LL_{7}=(0:1)\times(1:0)\times(u:u')$; and
 \item three lines $\LL_{8}=(s:s')\times(1:0)\times(1:0)$, $\LL_{9}=(s:s')\times(1:0)\times(0:1)$, $\LL_{10}=(s:s')\times(0:1)\times(0:1)$; 
 \item $7$ points of intersection of the previous lines: $\LL_1\cap \LL_5\cap \LL_8=\{(1:0)\times(1:0)\times(1:0)\}$, $\LL_1\cap \LL_6=\{(1:0)\times(0:1)\times(1:0)\}$, $\LL_2\cap \LL_5\cap \LL_9=\{(1:0)\times(1:0)\times(0:1)\}$, $\LL_2\cap \LL_6\cap \LL_{10}=\{(1:0)\times(0:1)\times(0:1)\}$, $\LL_3\cap \LL_7\cap \LL_8=\{(0:1)\times(1:0)\times(1:0)\}$, $\LL_4\cap \LL_7\cap \LL_9=\{(0:1)\times(1:0)\times(0:1)\}$ and $\LL_4\cap \LL_{10}=\{(0:1)\times(0:1)\times(0:1)\}$.
\end{enumerate}
Over those lines the fiber is of dimension $2$, except over the points of intersection of them.

In the language of Section $4.2$, we have that $W=\emptyset$. The set $\Theta$ formed by the sets $\alpha\subset [1,4]$ giving fibers of dimension $|\alpha|$, is 
\[
 \Theta=\{\{1,2\},\{1,3\},\{1,4\},\{2,3\},\{2,4\},\{3,4\},\{1,2,3\},\{1,3,4\},\{1,2,4\},\{2,3,4\}\}.
\]
Recall that this does not imply that every $\alpha\in \Theta$ will give an extraneous factor (cf.\ Remark \ref{rem-notHypersurface}). We clarify this:

As we have mentioned, the base locus is a union of lines with non-trivial intersection. Take $\alpha= \{1,2\}$. Set-theoretically $X_\alpha = \LL_1 \sqcup \LL_4$, and hence there are two irreducible components of $X_\alpha$, namely $X_{\alpha,1} = \LL_1$ and $X_{\alpha,2} = \LL_4$. The line $X_{\alpha,1} = \LL_1$ only intersects $\LL_5$, $\LL_6$ and $\LL_8$, hence 
\[
 \Omega_{\alpha,1} = \LL_1 \setminus (\LL_5\cap \LL_6\cap \LL_8)= \{(1:0)\times(t:t')\times(1:0)\ :\ t\neq 0 \textnormal{ and } t'\neq 0\}.
\]
\[
 \Omega_{\alpha,2} = \LL_4 \setminus (\LL_7\cap \LL_9\cap \LL_{10})= \{(0:1)\times(t:t')\times(0:1)\ :\ t\neq 0 \textnormal{ and } t'\neq 0\}.
\]
Since $\alpha= \{1,2\}$, the linear forms $L_1(p,\X)$ and $L_2(p,\X)$ vanish identically for all $p\in X_{\alpha}$, while $L_3(p,\X)=f_3(p)Y_3-g_3(p)X_3=t'Y_3$ and $L_4(p,\X)=tY_4$ for $p\in X_{\alpha,1}$. It is easy to note that none of them vanish if and only if $p\in \Omega_{\alpha,1}$. We get that $L_3(p,\X)=tX_3$ and $L_4(p,\X)=t'X_4$ for $p\in X_{\alpha,2}$.

Finally, for $\alpha= \{1,2\}$, we obtain two multiprojective bundles  $\EEE_{\alpha,i}$ over $ \Omega_{\alpha,i}$, for $i=1,2$,
\[
 \EEE_{\alpha,1} :\ \{(1:0)\times(t:t')\times(1:0)\times(\P1)^2\times(t':0)\times(t:0)\ :\ t\neq 0,\ t'\neq 0\} \nto{\pi_1}{\lto} \Omega_{\alpha,1},
\]
\[
 \EEE_{\alpha,2} :\ \{(0:1)\times(t:t')\times(0:1)\times(\P1)^2\times(0:t)\times(0:t')\ :\ t\neq 0,\ t'\neq 0\} \nto{\pi_1}{\lto} \Omega_{\alpha,2}.
\]
Observe that $\im(\phi_{\alpha, 1})=\PP^1\times\PP^1\times(1:0)\times(1:0)$, hence it does not define a hypersurface. Thus, $\phi_{\alpha, 1}$ does not contribute with an extraneous factor. The same for $\phi_{\alpha, 2}$. 

The situation is similar when $\alpha \in \{\{1,3\},\{1,4\},\{2,3\},\{2,4\}\}$, but quite different for $\alpha=\{3,4\}$. Take $\alpha=\{3,4\}$, the linear forms $L_3(p,\X)$ and $L_4(p,\X)$ vanish identically for all $p\in X_{\alpha}$. Take $X_{\alpha,1} = \LL_2$ and $X_{\alpha,2} = \LL_3$. Define $\Omega_{\alpha,1}:=\LL_3\setminus \{(0:1)\times(0:1)\times(1:0), (0:1)\times(1:0)\times(1:0)\}$, and observe that $\phi_{\alpha,1}:\Omega_{\alpha,1}\dto \PP_\alpha$ defines a hypersurface given by the equation $(X_2=0)$. Hence, when $\alpha=\{3,4\}$, $\phi_{\alpha,1}$ does give an extraneous factor.

Now, let us take $\alpha = \{1,2,3\}$ in order to illustrate a different situation. Verifying with the $7$ points listed above, we see that $X_\alpha= \{(1:0)\times(0:1)\times(1:0)\}\cup\{(0:1)\times(0:1)\times(0:1)\}$. Hence, there are two irreducible components $X_{\alpha,1}= \{(1:0)\times(0:1)\times(1:0)\}$ and $X_{\alpha,2}=\{(0:1)\times(0:1)\times(0:1)\}$, and clearly $\Omega_{\alpha,i}= X_{\alpha,i}$ for $i=1,2$. Thus, we get the trivial bundles 
\[
 \EEE_{\alpha,1} :\ \{(1:0)\times(1:0)\times(1:0)\times(\P1)^3\times(1:0)\ :\ t\neq 0 \textnormal{ and } t'\neq 0\} \nto{\pi_1}{\lto} \Omega_{\alpha,1},
\]
\[
 \EEE_{\alpha,2} :\ \{(0:1)\times(0:1)\times(0:1)\times(\P1)^3\times(0:1)\ :\ t\neq 0 \textnormal{ and } t'\neq 0\} \nto{\pi_1}{\lto} \Omega_{\alpha,2}.
\]
These two bundles give rise to the factors $Y_4$ and $X_4$. We conclude by a similar argument that the extraneous factor is
\[
 G= Y_1^2 X_2 Y_2 Y_3^2 X_4 Y_4.
\]
The degree of the multihomogeneous resultant $\Res_\Nc(L_1,L_2,L_3,L_4)$ in the coefficients of each $L_i$, as polynomials in $s,s',t,t',u$ and $u'$, is equal to $3\cdot 1\cdot 1\cdot 2=6$ for all $i=1,\hdots,4$ by \cite[Prop.\ 2.1, Ch.\ 13]{GKZ94}. So, the total degree of $\det((\k.)_\nu)$ is $24=4\cdot 6$. Indeed, the irreducible implicit equation is 
\[
 H= X_4^2 Y_1^2 Y_2^2 Y_3^2+2 X_4 X_2 X_3 Y_1^2 Y_2 Y_3 Y_4-X_4 X_1^2 X_3 Y_2^2 Y_3 Y_4+X_2^2 X_3^2 Y_1^2 Y_4^2,
\]
and $\deg(\phi)=2$. Thus, $\det((\k.)_\nu)=H^2\cdot G$ for $\nu \gg 0$.

\medskip

Let us change now our analysis, and consider the (smallest multiple of) the Newton polytope $\Nc$ of $f_i$ and $g_i$ for $i=1,2,3,4$. We easily see that $\Nc$ is a parallelepiped with opposite extremes in the points $(0,0,0)$ and $(1,1,2)$. For a suitable labeling of the points in $\Nc\cap \ZZ^3$ by $\{T_i\}_{i=0,\hdots,11}$, we have that the toric ideal that defines the toric embedding of $(\AA^\ast)^3 \nto{\iota}{\hto} \PP^{11}$ is

\noindent $J:= I(\Tc) = (T_9 T_{10}-T_8 T_{11}, T_7 T_{10}-T_6 T_{11}, T_5 T_{10}-T_{4} T_{11}, T_3 T_{10}-T_2 T_{11}, T_1 T_{10}-T_0 T_{11}, T^2_9-T_7 T_{11}, T_8 T_9-T_6 T_{11}, T_5 T_9- T_3 T_{11}, T_4 T_9  - T_2 T_{11}  , T_3 T_9  - T_1 T_{11}  , T_2 T_9  - T_0 T_{11}  , T^2_8  - T_6 T_{10}  , T_7 T_8  - T_6 T_9 , T_5 T_8  - T_2 T_{11}  , T_4 T_8  - T_2 T_{10}  , T_3 T_8  - T_0 T_{11}  , T_2 T_8  - T_0 T_{10}  , T_1 T_8 - T_0 T_9 , T_5 T_7  - T_1 T_{11}  , T_4 T_7  - T_0 T_{11}  , T_3 T_7  - T_1 T_9 , T_2 T_7  - T_0 T_9 , T_5 T_6  - T_0 T_{11}  , T_4 T_6  - T_0 T_{10}  , T_3 T_6  - T_0 T_9 , T_2 T_6  - T_0 T_8 , T_1 T_6  - T_0 T_7 , T_3 T_4  - T_2 T_5 , T_1 T_4  - T_0 T_5 , T^2_3  - T_1 T_5 , T_2 T_3  - T_0 T_5 , T^2_2  - T_0 T_4 , T_1 T_2  - T_0 T_3 )$.

\noindent This computation has been done in Macaulay2 using the code \cite{BotAlgo3D}.

The inclusion $\iota:(\AA^\ast)^3 \hto \PP^{11}$ defines a graded morphism of graded rings $\iota^\ast: \kk[T_0,\hdots,T_{11}]/J \to \kk[s,t,u]$. This morphism maps $T_1+T_{10}\mapsto f_1$, $T_7\mapsto g_1$, $T_4\mapsto f_2$, $T_7\mapsto g_2$, $T_6\mapsto f_3$, $T_5\mapsto g_3$, $T_0\mapsto f_4$, and $T_{11}\mapsto g_4$.

Hence, for $\alpha=\{1,2\}$, we have that
\[
 X_{\alpha}=\Proj (\kk[T_0,\hdots,T_{11}]/ (J + (T_1+T_{10},T_4,T_7))).
\]
Using Macaulay2, we can compute the primary decomposition of the radical ideal of $(T_1+T_{10},T_4,T_7)$ in $A:=\kk[T_0,\hdots,T_{11}]/ J$, obtaining the two irreducible components $X_{\alpha,1}$ and $X_{\alpha,2}$. Precisely,
\[
 X_{\alpha,1}=\Proj (\kk[T_0,\hdots,T_{11}]/ (J + (T_{10}, T_8, T_7, T_6, T_4, T_2, T_1, T_0))), \textnormal{ and}
\]
\[
 X_{\alpha,2}=\Proj (\kk[T_0,\hdots,T_{11}]/ (J + (T_{11}, T_7, T_6, T_5, T_4, T_1+T_{10}, T_0))).
\]
After embedding $(\PP^1)^3$ in $\PP^{11}$ via $\iota$, we get that $X_{\alpha,1}=\iota_\ast(L_1)$ and $X_{\alpha,2}=\iota_\ast(L_2)$ which coincides with the situation described above for $\Tc=\P1\times \P1\times \P1$.
\end{exmp}


\subsection*{Acknowledgments\markboth{Acknowledgments}{Acknowledgments}}
 I would like to thank my two advisors: Marc Chardin and Alicia Dickenstein, as well as Laurent Bus\'e and Marc Dohm for very useful discussions, ideas and suggestions. I am also grateful to the Galaad group at INRIA, for the always very kind hospitality. I would also like to thank the reviewer for his dedication and the very useful suggestions.

\def\cprime{$'$}

\end{document}